\documentclass{article}

\usepackage{dblfloatfix}

\usepackage{ifthen}
\newboolean{icml_template}
\setboolean{icml_template}{true} % false for preprint
\newcommand{\icml}[2]{\ifthenelse{\boolean{icml_template}}{#1}{#2}}

\newcommand{\RUNTITLE}{Sampling-based Nystr{\"o}m Approximation and
Kernel Quadrature}
\newcommand{\TITLE}{Sampling-based Nystr{\"o}m Approximation and Kernel Quadrature}

\icml{
\usepackage{microtype}
\usepackage{graphicx}
\usepackage{subfigure}
\usepackage{booktabs} % for professional tables

\usepackage{hyperref}

\usepackage[accepted]{icml2023} % modify to ICML 2023\

\usepackage{lipsum}
\usepackage{amsmath,amsthm,amssymb}
\usepackage{graphicx}
\usepackage{hhline,url}
\usepackage{mathrsfs}
\usepackage{xcolor}
\newtheorem{thm}{Theorem}[]
\newtheorem*{thm*}{Theorem}
\newtheorem{m-thm}[thm]{Meta-Theorem}
\newtheorem*{m-thm*}{Meta-Theorem}
\newtheorem{lem}{Lemma}[]
\newtheorem{remark}{Remark}[]
\newenvironment{rem}{\begin{remark}\rm}{\end{remark}}
\newtheorem{prop}{Proposition}[]
\newtheorem*{prop*}{Proposition}
\newtheorem{Definition}{Definition}

\newtheorem{Corollary}[]{Corollary}
\newenvironment{cor}{\begin{Corollary}}{\end{Corollary}}
\newtheorem{Example}[]{Example}

\newtheorem{algor}[thm]{Method}

\newtheorem{Condition}[thm]{Condition}

\newcommand{\Hil}{\mathcal{H}}

\newcommand{\R}{\mathbb{R}}
\newcommand{\dd}{\,\mathrm{d}}
\newcommand{\ve}{\varepsilon}

\newcommand{\F}{\mathcal{F}}

\renewcommand{\phi}{\varphi}

\newcommand{\X}{\mathcal{X}}
\newcommand{\bm}[1]{{\mbox{\boldmath $#1$}}}

\newcommand{\lmid}{\,\middle|\,}

\newcommand{\E}[1]{\mathbb{E}\!\left[#1\right]}
\renewcommand{\P}[1]{\mathbb{P}\!\left(#1\right)}
\newcommand{\ord}[1]{\mathcal{O}\!\left(#1\right)}
\newcommand{\K}{\mathcal{K}}

\newcommand{\ip}[1]{\left\langle #1 \right\rangle}

\newcommand{\vertiii}[1]{{\left\vert\kern-0.25ex\left\vert\kern-0.25ex\left\vert #1 
    \right\vert\kern-0.25ex\right\vert\kern-0.25ex\right\vert}}
\usepackage{caption}
\usepackage[inline]{enumitem}
\renewcommand{\tilde}{\widetilde}
\DeclareMathOperator{\wce}{wce}
\usepackage{bbm}

\usepackage{comment}
\usepackage{amsmath}
\usepackage{amssymb}
\usepackage{mathtools}
\usepackage{amsthm}

\usepackage[capitalize,noabbrev]{cleveref}

\theoremstyle{plain}

\theoremstyle{definition}

\usepackage[textsize=tiny]{todonotes}

\icmltitlerunning{\RUNTITLE}

\begin{document}

\twocolumn[
\icmltitle{\TITLE}





\begin{icmlauthorlist}
\icmlauthor{Satoshi Hayakawa}{o}
\icmlauthor{Harald Oberhauser}{o}
\icmlauthor{Terry Lyons}{o}
\end{icmlauthorlist}

\icmlaffiliation{o}{Mathematical Institute, University of Oxford,
Oxford, United Kingdom}

\icmlcorrespondingauthor{Satoshi Hayakawa}{hayakawa@maths.ox.ac.uk}

\icmlkeywords{Machine Learning, ICML}

\vskip 0.3in
]



\printAffiliationsAndNotice{}  % leave blank if no need to mention equal contribution
}{
    \input{others/header_neurips} % for preprint
}

\renewcommand{\H}{\mathcal{H}}
\newcommand{\N}{\mathcal{N}}
\newcommand{\mmd}{\mathrm{MMD}}
% modify later?

% ALREADY DONE: \begin{document}

\begin{abstract}
    We analyze the Nystr{\"o}m approximation
    of a positive definite
    kernel associated with a probability measure.
    We first prove an improved error bound for the conventional Nystr{\"o}m approximation with i.i.d.~sampling and singular-value decomposition in the continuous regime; the proof techniques are borrowed from statistical learning theory.
    We further introduce a refined selection of subspaces in
    Nystr{\"o}m approximation with theoretical guarantees
    that is applicable to non-i.i.d.~landmark points.
    Finally, we discuss their application to convex kernel quadrature
    and give novel theoretical guarantees as well as numerical observations.
\end{abstract}

\begin{table*}[!t]
    \centering
    \renewcommand{\arraystretch}{1.5}
    \caption{Main quantitative results. Individual bounds are available in
    Remark~\ref{rem:result}, Theorem~\ref{thm:k-s-mu-z-decay}, and Proposition~\ref{prop:emp-eig-decay-x}.
    For the explanation on each kernel, see at the end of {\bf Contribution} section.
    Here are remarks on the notation.
    (a) $\sigma_i$ is the $i$-th eigenvalue of the integral operator
    $\K:L^2(\mu)\to L^2(\mu); g\mapsto \int_\X k(\cdot, x)g(x)\dd\mu(x)$.
    (b) $\mu_X$ denotes the equally weighted empirical measure
    $\frac1N\sum_{i=1}^N\delta_{x_i}$ given by $X=(x_i)_{i=1}^N$.
    (c) $\mu(\cdot)$ and $\mu_X(\cdot)$ denote the integrals over the diagonal. See \eqref{eq:notation-measure}.
    }
    \label{table}
    \begin{tabular}{cccc}
        Quantity & Bound & Assumption \\\hline
        \begin{tabular}{c}
            $\mathbb{E}[\mu(\sqrt{k - k_s^Z})]$ \\
            $\mathbb{E}[\mu(k - k_s^Z)]$
        \end{tabular}
        & $\mathcal{O}\Biggl(
            \displaystyle\sqrt{\sum_{i>s}\sigma_i} + \frac{(\log\ell)^{2d+1}}\ell
            \Biggr)$
        & $\begin{cases}
            Z\sim_\mathrm{iid}\mu,\ \text{$k$: bounded} \\
            \sigma_i\lesssim \exp(-\beta i^{1/d})
        \end{cases}$
        \\
        \begin{tabular}{c}
            $\bigl( \mu(\sqrt{k^Z - \smash[b]{k_{s, \mu}^Z}})^2 \le \bigr)$
            $\mu(k^Z - k_{s, \mu}^Z)$ \\
            $\bigl(\mathbb{E}[\mu_X(\sqrt{k^Z - \smash[b]{k_{s, X}^Z}})]^2 \le \bigr)$
            $\mathbb{E}[\mu_X(k^Z - k_{s, X}^Z)]$
        \end{tabular}
         & $\displaystyle\sum_{i>s}\sigma_i$ &
        \begin{tabular}{c}
             $Z$: fixed\\
             $Z$: fixed, $X\sim_\mathrm{iid}\mu$
        \end{tabular}
    \end{tabular}
\end{table*}

\section{Introduction}
Kernel methods form a prominent part among modern machine learning tools. 
However, making kernel methods scalable to large datasets is an ongoing challenge. 
The main bottleneck is that the kernel Gram matrix scales quadratically in the number of data points. 
For large scale problems the number of matrix entries can easily be of the order hundred-thousands or millions so that even storing the full Gram matrix can become too costly. 
Several approaches have been developed to deal with these, among the most prominent are the Random Fourier Features and the Nystr\"om method. 
In this article, we revisit and generalize the Nystr{\"o}m method and provide new error estimates.
Consequences are theoretical guarantees for kernel quadrature and improvements on the standard Nystr{\"o}m method that go beyond uniform subsampling of data points.

\paragraph{Nystr{\"o}m Approximation.}
The main idea of the Nystr{\"o}m method is to replace the original kernel $k$ by another kernel $k_\mathrm{app}$ that is constructed by random projection of the elements in the (in general infinite-dimensional) RKHS associated with $k$ into a low-dimensional RKHS.
A consequence of this is that the Gram matrix of $k_\mathrm{app}$ is a low-rank approximation of the original Gram matrix. 
Concretely, let $\mu$ denote a probability measure on a (Hausdorff) space $\X$ and $k$ a kernel on $\X$; then the standard Nystr{\"o}m approximation uses the random kernel 
\begin{equation}
    k^Z(x, y) := k(x, Z)k(Z, Z)^+k(Z, y).
    \label{eq:def-k-z}
\end{equation}
where $Z=(z_i)_{i=1}^\ell$ is an $\ell$-point subset of $\X$ usually taken i.i.d.~from $\mu$
\citep{dri05,kum12}.

\paragraph{Further $\boldsymbol{s}$-rank Approximation.}
While less common,
the following rank-reduced version is of our interest:
\begin{align}\label{eq:nys}
    k_\mathrm{app}(x,y)= k^Z_s(x, y):= k(x, Z)k(Z, Z)^+_sk(Z, y),
\end{align}
where $k(Z, Z)^+_s$ is the Moore--Penrose pseudo-inverse of the best $s$-rank approximation of the Gram matrix
$k(Z, Z)=(k(z_i, z_j))_{i,j=1}^\ell$ with $s\le \ell$.
Note that $k_\ell^Z = k^Z$.

Our motivation for this rank reduction comes from kernel-based numerical integration.
Indeed, if we are given an $s$-rank kernel $k_\mathrm{app}$ and a probability measure $\mu$,
by Tchakaloff's theorem
there is a discrete probability measure $\nu$ supported over at most $s+1$ points
satisfying $\int_\X f\dd\mu = \int_\X f\dd\nu$ for all $f\in\H_{k_\mathrm{app}}$,
where $\H_{k_\mathrm{app}}$ is the finite-dimensional RKHS associated with the kernel $k_\mathrm{app}$.
Such a measure $\nu$ works as a kernel quadrature rule if the $k_\mathrm{app}$ well approximates the original kernel $k$,
and the rank $s$ directly affects the number of (possibly expensive) function evaluations we need to estimate each integral.
The primary error criterion in this paper is
\begin{align}
&\int_\X\sqrt{k(x, x) - k_\mathrm{app}(x, x)}\dd\mu(x)\label{eq:sqrt error},
%&\int_\X({k(x, x) - k_\mathrm{app}(x, x)})\dd\mu(x)
\end{align}
which arises from the error estimate in kernel/Bayesian quadrature
\citep{hayakawa21b,ada22}.

\paragraph{Contribution.}
Our first theoretical result is that the expectation of \eqref{eq:sqrt error} is of the order 
$\ord{\sqrt{\sum_{i>s}\sigma_i} + \mathrm{polylog}(\ell)/\ell}$
when the eigenvalues $(\sigma_i)_{i=1}^\infty$ of the kernel integral operator induced by $(k,\mu)$ enjoy exponential convergence (the expectation is taken over the empirical sample $Z$).
Key to the proof of this bound is the use of concepts from statistical learning theory; in particular, the (local) Rademacher complexity.
This error estimate is far better than the bound $\ord{\text{spectral term} + s^{1/2}/\ell^{1/4}}$
that follows from the existing high-probability estimate
$\int_\X (k(x, x)-k_s^Z(x, x))\dd\mu(x)=\mathcal{O}(s\sigma_s + \sum_{i>s}\sigma_i
+ s/\sqrt{\ell})$
\citep[][Corollary 4]{hayakawa21b}.
By combining our new bound with known kernel quadrature estimates this explains the strong empirical performance of the random kernel quadrature, see \citet{hayakawa21b}; previously the theoretical bounds were not even better than Monte-Carlo
in terms of $\ell$.

Our second contribution
is the use of other $k_\mathrm{app}$ than $k_s^Z$
with better bounds of \eqref{eq:sqrt error},
for a general class of landmark points $Z$
rather than just an i.i.d.~sample from $\mu$.
This generalization allows to use other sets $Z$ in \eqref{eq:nys} to
achieve better overall performance; e.g.~sampling $Z$ from determinantal point processes (DPPs) on $\X$ is known to be advantageous in applications.
To construct and provide theoretical guarantees for such improved Nystr{\"o}m constructions we revisit and generalize a method that was proposed in \citet{san16} and give further theoretical guarantees
applicable to kernel quadrature rules.

The following is the list of low-rank approximations presented in the paper:
\begin{itemize}
    \item $k^Z$ and $k^Z_s$: Usual Nystr{\"o}m approximations using landmark points $Z$.
    See \eqref{eq:def-k-z} and \eqref{eq:nys}.
    \item $k_{s,\mu}^Z$: The $s$-rank truncated Mercer decomposition of
    the kernel $k^Z$ with respect to the measure $\mu$.
    See \eqref{eq:k-s-mu-formal}.
    \item $k_{s, X}^Z$: A version of $k_{s,\mu}^Z$ with $\mu$
    given by the empirical measure $\frac1N
    \sum_{i=1}^N \delta_{x_i}$ of the set $X=(x_i)_{i=1}^N$.
    This actually coincides with $k_s^Z$ when $X=Z$; see \eqref{eq:mu-z-mercer-empirical}.
\end{itemize}
See Table~\ref{table} for a summary of our quantitative results.

\paragraph{Outline.}
Section \ref{sec:literature} discusses the existing literature and introduces some notation.
Section \ref{sec:main estimate} contains our first main result,
namely the analysis of $k_s^Z$ for an i.i.d.~$Z$;
Appendix \ref{sec:stat-learn} provides the necessary background from statistical learning theory. 
In Section \ref{sec:new LR}, we then treat a general $Z$
to give refined low-rank approximations together with theoretical guarantees, rather than the conventional $k_s^Z$.
In Section \ref{sec:kernel quadrature}, we discuss how our bounds yields
new theories and methods for the recent random kernel quadrature construction, which enables us to explain the empirical performance
as well as to build some strong candidates
whose performance is assessed by numerical experiments.
All the omitted proofs are given in Appendix~\ref{sec:proofs}.

\section{Related Literature and Notation}\label{sec:literature}
To simplify the notation, we denote
\begin{align}
    \nu(f) := \int_\X f(x)\dd\nu(x),
    \ \ \ 
    \nu(h) := \int_\X h(x, x)\dd\nu(x)
    \label{eq:notation-measure}
\end{align}
for any functions $f:\X\to\R$, $h:\X\times\X\to\R$ and a (probability) measure $\nu$ on $\X$, whenever the integrals are well-defined.
In this notation, the aim of this paper is to bound
$\mu(\sqrt{k - k_\mathrm{app}})$
or $\mu(k - k_\mathrm{app})$
for a class of low-rank approximation $k_\mathrm{app}$.
Also, $A^+$ denotes the Moore--Penrose pseudo-inverse of a matrix $A$.

\paragraph{Approximation of the Gram Matrix.}
The standard use of the Nystr{\"o}m method in ML is to replace the Gram matrix $k(X, X)$ for a set $X = (x_i)_{i=1}^N$ by the low-rank matrix $k^Z(X,X)$ where $k^Z$ is defined as in \eqref{eq:def-k-z}.
A well-developed literature studies the case when $Z=(z_i)_{i=1}^\ell$ is uniformly and independently sampled from $X$, see \citet{dri05,kum12,yan12,jin13,li15}. 
Further, the cases of leverage-based sampling \citep{git16}, DPPs \citep{li16}, and kernel $K$-means samples \citep{ogl17} have received attention. 
Moreover, two variants of the standard Nystr{\"o}m method have been studied:
the first replaces the Moore-Penrose inverse of $k(Z,Z)$ in \eqref{eq:def-k-z} with the pseudo-inverse of the best $s$-rank approximation of $k(Z, Z)$ as in \eqref{eq:nys} via SVD \citep{dri05,kum12,li15}; the second uses the best $s$-rank approximation of $k^Z(X,X)$, see \citep{tro17,wan19}.
For a brief overview in this regard, see \citet[][Remark 1]{wan19}.
\paragraph{Approximation of the Integral Operator.}
The matrix $k(X,X)$ can be regarded as a finite-dimensional representation of the linear (integral) operator
\[
  \K: L^2(\mu)\to L^2(\mu),\quad  (\K f)(x) = \int_\X k(x, y)f(y)\dd\mu(y).
\]
We denote with $(\sigma_i, e_i)_{i=1}^\infty$ the eigenpairs of the operator $\K$, and assume the eigenvalues are ordered $\sigma_1\ge\sigma_2\ge\cdots\ge0$.
The Mercer decomposition exists under mild assumptions (for example, $\mathop\mathrm{supp}\mu=\X$, $k$ is continuous and $\int_\X k(x, x)\dd\mu(x)<\infty$ \citep{ste12} are sufficient) and gives the representation
\begin{equation}
    k(x, y) = \sum_{i=1}^\infty \sigma_i e_i(x)e_i(y),
    %\qquad \text{$\mu$-a.s. $x, y\in\X$,}
    \label{eq:mu-mercer}
\end{equation}
where $\lVert e_i\rVert_{L^2(\mu)}=1$, and $(\sqrt{\sigma_i}e_i)_{i=1}^\infty$ is an orthonormal basis of the RKHS $\H_k$ of $k$.
Hence, a natural approach is to just truncate this expansions after $s$ terms, $k_\mathrm{app}=\sum_{i=1}^s\sigma_ie_i(x)e_i(y)$, to get a finite-dimensional approximation of the kernel $k$.
This approach is natural since the approximation quality of the operator $\K$ determines the resulting error estimates. 
Unfortunately, it is often rendered useless since the Mercer decomposition depends on the tuple $(k,\mu)$ and while explicit expression are known for special choices, in general it is unlikely to have a closed-form representation of the eigenpairs $(\sigma_i, e_i)_{i=1}^\infty$.

\paragraph{Other Approximations.}
A compromise which is relevant to our work is proposed in  \citet{san16}.
Instead of using the Mercer decomposition of $\K$ one uses the Mercer decomposition of \eqref{eq:def-k-z}.
Our main result allows to generalize this approach and to provide theoretical guarantees missing in the reference.
Related is the article \citet{gau21} that studies the interactions of
several Hilbert-Schmidt spaces of (integral) operators
given by a Nystr{\"o}m approximation/projection of a kernel-measure pair
as in the present paper; further, \citet{cha22} considers a low-rank approximation of an empirical kernel mean embedding by using a Nystr{\"o}m-based projection.
The leverage-based sampling studied in \citet{git16}
has continuous counterparts. One with a slight modification is in the kernel literature \citep{bac17}, while the exact counterpart can be found in a context from approximation theory \citep{coh17} under the name of {\it optimally-weighted sampling}, which essentially proposes sampling from $s^{-1}\sum_{i=1}^se_i^2(x)\dd\mu(x)$.
\paragraph{The Power Function.}
Finally, the square root of the diagonal term
$\sqrt{k(x, x) - k^Z(x, x)}$ or its generalization
is known as the {\it power function} in the literature on kernel-based interpolation \citep{dem03,san17,kar21}.
There the primary interest is its $L^\infty$ (uniform) norm, rather than the $L^1(\mu)$ norm, $\mu(\sqrt{k - k_\mathrm{app}})$,
or the $L^2(\mu)$ norm, $\mu(k - k_\mathrm{app})$, that appear in kernel quadrature estimates and error estimates of the Nystr{\"o}m/Mercer type decompositions.
\paragraph{Kernel Quadrature.}
The literature on kernel quadrature includes
herding \citep{che10,bac12,hus12,tsu22},
weighted/correlated sampling \citep{bac17,bel19,bel20,bel21},
a subsampling method called thinning \citep{dwi21,dwi22,she22}
and a positively weighted kernel quadrature \citep{hayakawa21b}
that motivated our work.
We refer to \citet[][Table 1]{hayakawa21b} for comparison of existing algorithms in terms of their convergence guarantees
and computational complexities.

\section{Analyzing $k_s^Z$ for i.i.d.~$Z$ via Statistical Learning Theory}\label{sec:main estimate}
Let $Z=(z_i)_{i=1}^\ell\subset\X$
and $k^Z_s$ be the $s$-dimensional kernel given by
$k^Z_s(x, y) = k(x, Z)k(Z, Z)_s^+k(Z, y)$ as in the usual Nystr{\"o}m approximation.
Throughout the paper,
suppose we are provided the singular value decomposition of the matrix
$k(Z, Z) = U\mathop\mathrm{diag}(\lambda_1, \ldots,\lambda_\ell)U^\top$
with an orthogonal matrix $U = [u_1, \ldots, u_\ell]$
and $\lambda_1\ge\cdots\ge\lambda_\ell\ge0$.
Note that 
\begin{equation}
    k_s^Z(x, y) = \sum_{i=1}^s \bm{1}_{\{\lambda_i>0\}}\frac1\lambda_i
    (u_i^\top k(Z, x))(u_i^\top k(Z, y))
    \label{eq:mu-z-mercer-empirical}
\end{equation}
is actually a truncated Mercer decomposition of $k^Z$ with regard to the measure
$\mu_Z = \frac1\ell\sum_{i=1}^\ell\delta_{z_i}$,
since
\begin{align*}
    &\ip{u_i^\top k(Z, \cdot), u_j^\top k(Z, \cdot)}_{L^2(\mu_Z)}\\
    &= \frac1\ell u_i^\top k(Z, Z) k(Z, Z)u_j = \frac{\lambda_i\lambda_j}\ell \delta_{ij}.
\end{align*}
This fact is at the heart of our analysis:
$k_s^Z$ is `optimal' $s$-rank approximation for the measure $\mu_Z$,
and the statistical learning theory connects estimates in empirical measure and the original measure.

Let us denote by $P_{Z, s}:\H_k\to\H_k$
the linear operator given by $k(\cdot, x) \mapsto k^Z_s(\cdot, x)$
for all $x\in\X$.
We shall also simply write $P_Z = P_{Z,\ell}$.
\begin{lem}\label{lem:proj-trivial}
    $P_{Z, s}$ is an orthogonal projection in $\H$.
\end{lem}

This projection is related the quantity of interest, in that
$k_s^Z(x, x) = \ip{k(\cdot, x), P_{Z, s} k(\cdot, x)}_{\H_k}
= \lVert P_{Z,s}k(\cdot, x)\rVert_{\H_k}^2$.
Thus, we have
$
    k(x, x) - k_s^Z(x, x)
    = \lVert P_{Z,s}^\perp k(\cdot, x)\rVert_{\H_k}^2
$
by using $P_{Z, s}^\perp$, the orthogonal complement of $P_{Z, s}$.
So we are now interested in estimating the integral
$\mu(\sqrt{k-k_s^Z})
= \int_\X \lVert P_{Z,s}^\perp k(\cdot, x)\rVert_{\H_k}\dd\mu(x)$
from the viewpoint of the projection operator.
We first estimate its empirical counterpart $
\mu_Z(\sqrt{k - k_s^Z}) = \frac1\ell\sum_{i=1}^\ell
\lVert P_{Z,s}^\perp k(\cdot, z_i)\rVert_{\H_k}$,
where $\mu_Z = \frac1\ell\sum_{i=1}^\ell \delta_{z_i}$
is the empirical measure.
Indeed, we have the following identity regarding $\mu_Z(k-k_s^Z)$:

\begin{lem}\label{lem:emp}
    For any $\ell$-point sample $Z\subset \X$,
    we have
    \[
        \mu_Z(\sqrt{k - k_s^Z})^2
        \le \mu_Z(k-k_s^Z)
        = \displaystyle\frac1\ell\sum_{i=s+1}^\ell \lambda_i
        % \left(\frac1\ell\sum_{i=1}^\ell
        % \lVert P_{Z,s}^\perp k(\cdot, z_i)\rVert_{\H_k}\right)^2
        % \le \frac1\ell\sum_{i=1}^\ell
        % \lVert P_{Z,s}^\perp k(\cdot, z_i)\rVert_{\H_k}^2
        % = \frac1\ell\sum_{j=s+1}^\ell \lambda_j,
    \]
    where $\lambda_1\ge\cdots\ge\lambda_\ell$ are
    eigenvalues of $k(Z, Z)$.
\end{lem}

When $Z$ is given by an i.i.d.~sampling,
the decay of eigenvalues $\lambda_i$ enjoys the rapid decay
given by $\sigma_i$ in the following sense:
\begin{lem}\label{lem:trace}
    Let $Z=(z_i)_{i=1}^\ell$ be an $\ell$-point independent sample from $\mu$.
    Then, for the eigenvalues $\lambda_1\ge\cdots\ge \lambda_\ell$ of $k(Z, Z)$,
    we have
    \[
        \E{\frac1\ell\sum_{i=s+1}^\ell\lambda_i}\le \sum_{i>s} \sigma_i.
    \]
\end{lem}

For a general random orthogonal projection operator,
we can prove the following bound by using arguments in statistical learning theory
(Section \ref{sec:stat-learn}):
\begin{thm}\label{thm:low-dim-proj}
    Let $Z = (z_i)_{i=1}^\ell$ be an $\ell$-point
    independent sample from $\mu$
    and $P$ be a random orthogonal projection in $\H_k$
    possibly depending on $Z$.
    For any integer $m\ge1$, we have the following bound:
    \begin{align*}
        \E{\int_\X \lVert Pk(\cdot, x)\rVert_{\Hil_k} \dd\mu(x)}
        \le \E{\frac2\ell\sum_{i=1}^\ell \lVert Pk(\cdot, z_i)\rVert_{\Hil_k}} & \\
        +\ 4\sqrt{\sum_{i>m}\sigma_i} 
        + \frac{\sqrt{k_{\max}}}{\ell}\left(\frac{80m^2\log(1+2\ell)}9 + 69\right)&,
    \end{align*}
    where the expectation is taken regarding the draws of $Z$.
\end{thm}

Recall that $\mu(\sqrt{k-k_s^Z})=\int_\X \lVert
P_{Z, s}^\perp k(\cdot, x)\rVert_{\H_k}\dd\mu(x)$.
By combining this theorem when $P=P_{Z, s}^\perp$
and Lemma \ref{lem:emp} \& \ref{lem:trace},
we can obtain the following:
\begin{cor}\label{cor:wce-iid}
    Let $Z=(z_i)_{i=1}^\ell$ be an $\ell$-point independent sample from $\mu$.
    Then, for any integer $m\ge1$,
    we have
    \begin{align*}
        \E{\mu(\sqrt{k - k_s^Z})}
        &\le 2\sqrt{\sum_{i>s}\sigma_i}
        + 4\sqrt{\sum_{i>m}\sigma_i}\\
       &\quad + \frac{\sqrt{k_{\max}}}{\ell}\left(\frac{80m^2\log(1+2\ell)}9 + 69\right).
    \end{align*}
\end{cor}

\begin{rem}\label{rem:result}
    When $\sigma_j\lesssim e^{-\beta i^{1/d}}$ with a constant
    $\beta>0$ and a positive integer $d$
    \citep[typical for $d$-dimensional Gaussian kernel, see, e.g.,][Section A.2]{ada22},
    by taking $m \sim (\log\ell)^d$, we have a bound
    \[
        \E{\mu(\sqrt{k-k_s^Z})}
        = \ord{ \sqrt{\sum_{i>s}\sigma_i} + \frac{(\log\ell)^{2d+1}}\ell}
    \]
    for $\ell\ge3$; see Appendix~\ref{sec:proof-rem} for the proof.
    Since $k-k_s^Z \le \sqrt{k_{\max}}\sqrt{k-k_s^Z}$, the same estimate applies to
    $\mathbb{E}[\mu(\sqrt{k-k_s^Z})]$.
    These also lead to an $(s+1)$-point randomized convex kernel quadrature $Q_{s+1}$ with the same order of $\E{\wce(Q_{s+1})}$.
    See Section \ref{sec:kernel quadrature} for details.
\end{rem}

\section{A Refined Low-rank Approximation with General $Z$}\label{sec:new LR}
The process of obtaining a good approximation $k_\mathrm{app}$
of $k$ using $k^Z$ can be decomposed into two parts:
\[
    k - k_\mathrm{app} = \underbrace{k - k^Z}_{\mathrm{A}} + \underbrace{k^Z - k_\mathrm{app}}_{\mathrm{B}}.
\]
In the previous section, we have analyzed the case $Z$ is i.i.d.~and
$k_\mathrm{app}=k_s^Z$.
However, we can consider
more general $Z$,
and indeed we actually have a better way to select a subspace
(i.e., $k_\mathrm{app}$)
from the finite-rank kernel $k^Z$
rather than just using $k_s^Z$.

\subsection{Part A: Estimating the Error of $k^Z$ for General $Z$}
This part is relatively well-studied.
Indeed, $\mu(k - k^Z) = \int_\X (k(x, x) - k^Z(x, x))\dd\mu(x)$
for some non-i.i.d.~$Z$ can be bounded by using the results
of weighted kernel quadrature.
For example, \citet{bel19}
consider the worst-case error for the weighted integral
\begin{equation}
    \mu(fg)=\int_\X f(x)g(x)\dd\mu(x) \approx \sum_{i=1}^\ell w_if(z_i)
    \label{eq:kq-weighted}
\end{equation}
for any $\lVert f \rVert_{\H_k} \le 1$ and a fixed $g\in L^2(\mu)$
with $Z=(z_i)_{i=1}^\ell$ following a certain DPP.
Now consider the optimal worst-case error in the above approximation
for the fixed point configuration $Z$:
\begin{align}
    \icml{ & }{}
    \inf_{w_i}\sup_{\lVert f\rVert_{\H_k}\le1}\left\lvert \mu(fg)-
    \sum_{i=1}^\ell w_if(z_i)\right\rvert
    \icml{ \nonumber\\ }{}
    &= \sup_{\lVert f\rVert\le1}
    \left\lvert\ip{
        f, \int_\X k(\cdot, x)g(x)\dd\mu(x) - \sum_{i=1}^\ell w_ik(\cdot, z_i)
    }_{\H_k}\right\rvert \nonumber \\
    &= \inf_{w_i}\left\lVert
        \K g - \sum_{i=1}^\ell w_i k(\cdot, z_i)
    \right\rVert_{\H_k} = \lVert P_Z^\perp \K g\rVert_{\H_k}.
    \label{eq:weighted-wce}
\end{align}
By using this, we can prove the following estimate:
\begin{prop}\label{prop:wce-decomp}
    For any finite subset $Z\subset \X$ and any integer $m\ge0$,
    we have
    \begin{align*}
        %\int_X \lVert P_Z^\perp k(\cdot, x)\rVert_{\H_k}^2\dd\mu(x)
        \mu(k - k^Z)
        = \sum_{i=1}^\infty \lVert P_Z^\perp \K e_i\rVert_{\H_k}^2
        \le \sum_{i=1}^m \lVert P_Z^\perp \K e_i\rVert_{\H_k}^2
        + \sum_{i>m} \sigma_i
    \end{align*}
    where $(\sigma_i, e_i)_{i=1}^\infty$ are the eigenpairs of $\K$.
\end{prop}

The papers \citet{bel19,bel20,bel21}
give bounds on the worst-case error of the weighted kernel quadrature
\eqref{eq:weighted-wce}
when $Z$ is given by some correlated sampling,
whereas \citet{bac17} gives another bound
when $Z$ is given by an optimized weighted sampling rather than sampling from $\mu$.
By using \eqref{eq:weighted-wce} and Proposition \ref{prop:wce-decomp},
we can import their bounds on
weighted kernel quadrature with non-i.i.d.~$Z$
to the estimate of $\mu(k - k^Z) = \int_X \lVert P_Z^\perp k(\cdot, x)\rVert_{\H_k}^2\dd\mu(x)$.
Here, we just give one such example:
\begin{cor}\label{cor:dpp}
    Let $Z = (z_i)_{i=1}^\ell$ be taken from a DPP
    given by the projection kernel $p(x, y) = \sum_{i=1}^\ell e_i(x)e_i(y)$
    with a reference measure $\mu$,
    i.e., $\P{Z\in A} = \frac1{\ell!}\int_A\det p(Z, Z)\dd\mu^{\otimes\ell}(Z)$
    for any Borel set $A\subset\X^d$.
    Then, for any integer $m\ge0$, we have
    \[
        \E{\mu(k - k^Z)}
        \le \sum_{i>m}\sigma_i + 4m\sum_{i>\ell}\sigma_i,
    \]
    where the expectation is taken regarding the draws of $Z$.
\end{cor}

In any case, by using those non-i.i.d.~points,
we can obtain a better $Z$ in the
sense that $\int_\X (k(x, x) - k^Z(x, x))\dd\mu(x)$
attains a sharper upper bound than 
the bound given in the previous section for
an $\ell$-point i.i.d.~sample from $\mu$.
However, for a general $Z$,
it is not necessary sensible to execute the SVD of $k(Z, Z)$ and get $k_s^Z$ accordingly,
as a SVD of $k(Z, Z)$ corresponds
to approximating $\mu$ by the empirical measure $\frac1\ell\sum_{i=1}^\ell \delta_{z_i}$
(indeed, this observation is the key to the results in the previous section).
Thus, for points $Z$ not given by i.i.d. sampling,
there should exist a better choice of $k_\mathrm{app}$ than $k_s^Z$.
We discuss this in the following section.

\subsection{Part B: Mercer Decomposition of $k^Z$}
Instead of using $k_s^Z$, we propose to compute the Mercer decomposition of $k^Z$
with respect to $\mu$
and truncate it to get $k_{s, \mu}^Z$, which is defined in the following.
This is doable if we have knowledge of $h_\mu(x, y):=\int_\X k(x, t)k(t, y)\dd\mu(t)$,
since $k^Z$ is a finite-dimensional kernel.
We can prove the following:
\begin{lem}\label{rem:square-kernel}
We have $h_\mu(x, y) = \sum_{i=1}^\infty
\sigma_i^2 e_i(x)e_i(y)$.
\end{lem}
    
We now discuss how $h_\mu$ can be used to derive the Mercer decomposition of $k^Z$.
Note that this can be regarded as a generalization of \citet[][Section 6]{san16}.
Let $\K^Z:L^2(\mu)\to L^2(\mu)$ be the integral operator given by
$g\mapsto \int_\X k^Z(\cdot, x)g(x)\dd\mu(x)$.

For functions of the form $f=a^\top k(Z, \cdot)$
and $g =  b^\top k(Z, \cdot)$
with $a, b\in\R^\ell$,
we have
\begin{align}
    \ip{f, g}_{L^2(\mu)} &=
    \int_\X a^\top k(Z, x)k(x, Z)b\dd\mu(x)
    \icml{ \nonumber\\& }{}
    = 
    a^\top h_\mu(Z, Z) b.
    \label{eq:ip-z}
\end{align}
So, if we write $h_\mu(Z, Z)=H^\top H$ by using an $H\in\R^{\ell\times\ell}$
(since $h_\mu(Z, Z)$ is positive semi-definite),
an element $f = a^\top k(Z, \cdot)\in L^2(\mu)$
is non-zero if and only if $Ha\ne 0$.
Furthermore, we have
\begin{align}
    \mathcal{K}^Zf
    &= \int_{\X} k(\cdot, Z)k(Z, Z)^+k(Z, x)k(x, Z)a
        \dd\mu(x)\nonumber\\
    &= k(\cdot, Z) k(Z, Z)^+ h_\mu(Z, Z)a \nonumber\\
    &= \left[k(Z, Z)^+ h_\mu(Z, Z)a\right]^\top k(Z, \cdot).
    \label{eq:k-z}
\end{align}
Thus, $f$ is a nontrivial eigenfunction of $\K^Z$,
if $Ha\ne0$
and $a$ is an eigenvector of $k(Z, Z)^+ h_\mu(Z, Z)$.
It is equivalent to $c = Ha$ being an eigenvector of
$Hk(Z, Z)^+H^\top$.

Let us decompose this matrix
by SVD as $H k(Z, Z)^+ H^\top = V\mathop\mathrm{diag}(\kappa_1, \ldots, \kappa_\ell)V^\top$,
where the $V = [v_1, \ldots, v_\ell]\in\R^{\ell\times\ell}$ is an orthogonal matrix
and $\kappa_1\ge\cdots\ge\kappa_\ell\ge0$.
Then, we have
\[
    Hk(Z, Z)^+H^\top = \sum_{i=1}^\ell \kappa_i v_iv_i^\top.
\]
Let us consider $f_i = (H^+v_i)^\top k(Z, \cdot) = v_i^\top (H^+)^\top k(Z, \cdot)$
for $i=1,\ldots, \ell$
as candidates of eigenfunctions of $\K^Z$.
We can actually prove the following:
\begin{lem}\label{lem:or-eig}
    The set $\{f_i\mid i\ge1,\,\kappa_i>0\}$ forms an orthornomal subset of
    $L^2(\mu)$
    whose elements are eigenfunctions of $\K^Z$.
\end{lem}

Let us define $k^Z_\mu(x, y):= \sum_{i=1}^\ell \kappa_i f_i(x)f_i(y)$;
note that this is computable.
From the above lemma,
this expression is a natural candidate for ``Mercer decomposition'' of $k^Z$.
We can prove that it actually coincides with $k^Z(x, y)$ $\mu$-almost everywhere,
and so the decomposition is independent of the choice of $H$ up to $\mu$-null sets:
\begin{prop}\label{prop:ae}
    There exists a measurable set $A\subset\X$ depending on $Z$
    with
    $\mu(A)=1$ such that
    $k^Z(x, y) = k^Z_\mu(x, y)$
    holds for all $x,y\in A$.
    Moreover, we can take $A = \X$ if $\ker h_\mu(Z, Z)\subset \ker k(Z, Z)$.
\end{prop}

Now we just define $k_{s,\mu}^Z$ for $s\le\ell$ as follows:
\begin{equation}
    k_{s,\mu}^Z(x, y):= \sum_{i=1}^s \kappa_if_i(x)f_i(y).
    \label{eq:k-s-mu-formal}
\end{equation}
\begin{thm}\label{thm:k-s-mu-z-decay}
    We have
    %\int_\X (k^Z_\mu(x, x) - k_{s,\mu}^Z(x, x))\dd\mu(x) 
    $\mu(k_\mu^Z - k_{s,\mu}^Z) \le \sum_{i=s+1}^\ell\sigma_i$
    for any $Z = (z_i)_{i=1}^\ell \subset\X$.
\end{thm}
\begin{proof}
    The left-hand side is equal to $\sum_{i=s+1}^\ell\kappa_i$
    from Lemma \ref{lem:or-eig} and the definition of the kernels.
    Thus, it suffices to prove $\kappa_i\le\sigma_i$ for each $i$.
    It directly follows from the min-max principle (or Weyl's inequality) as
    $k - k_\mu^Z$ is positive definite on an $A\subset\X$
    with $\mu(A)=1$ from Proposition \ref{prop:ae}.
\end{proof}

\begin{rem}
    The choice of the matrix $H$ with
    $H^\top H = h_\mu(Z, Z)$
    does not affect the theory but might affect the numerical errors.
    We have used the matrix square-root $h_\mu(Z, Z)^{1/2}$,
    i.e., the symmetric and positive semi-definite matrix $H$ with
    $H^2 = h_\mu(Z, Z)$, throughout the experiments in Section~\ref{sec:kernel quadrature}, so that we just need to take the pseudo-inverse of
    positive semi-definite matrices.
\end{rem}

\paragraph{Approximate Mercer Decomposition.}
When we have no access to the function $h_\mu$,
we can just approximate it by using an empirical measure.
For a $X=(x_j)_{j=1}^M\subset\X$,
denote by $h_X$ the function given by
replacing $\mu$ in $h_\mu$ with the empirical measure with points $X$:
\[
    h_X(x, y) = \frac1M\sum_{j=1}^M k(x, x_j)k(x_j, y)=\frac1M k(x, X)k(X, y).
\]
We can actually replace every $h_\mu$ by $h_X$ in the above construction to
define $k_X^Z$ and $k_{s, X}^Z$.
This approximation is already mentioned by \citet{san16} 
without theoretical guarantee.
Another remark is that, when restricted on the set $X$,
it is equivalent to the best $s$-rank approximation of $k^Z(X, X)$
in the Gram-matrix case \citep{tro17,wan19},
since the $L^2$-norm for the uniform measure on $X$
just corresponds to the $\ell^2$-norm in $\R^{\lvert X\rvert}$.

Note that we have $k_X^Z(X, X) = k^Z(X, X)$ from Proposition \ref{prop:ae}
in the discrete case.
As we have $\ker h_X(Z, Z) = \ker k(Z, X)k(X, Z) = \ker k(X, Z)$,
we additionally obtain the following sufficient condition
from Proposition \ref{prop:ae}.
\begin{prop}\label{prop:discrete}
    $k_X^Z(x, y)= k^Z(x, y)$ holds for all $x, y\in X$. 
    Moreover, if $\ker k(X, Z)\subset \ker k(Z, Z)$,
    then we have $k_X^Z = k^Z$ over the whole $\X$.
\end{prop}
In particular, we have $k_X^Z = k^Z$ whenever $Z\subset X$.
These (at least $\mu$-a.s.) equalities
given in Proposition \ref{prop:ae} \& \ref{prop:discrete}
are necessary for the applications to kernel quadrature,
since we need $k - k_\mathrm{app}$ to be positive definite
for exploiting the existing guarantees such as 
Theorem~\ref{thm:pkq} in the next section.

Although checking $k_X^Z=k^Z$ is not an easy task,
from the first part of Proposition \ref{prop:discrete},
$k_{s,X}^Z$ satisfies the following estimate
in terms of the empirical measure $\mu_X$.
\begin{prop}\label{prop:emp-eig-decay-x}
    Let $Z\subset \X$ be a fixed subset
    and $X$ be an $M$-point independent sample from $\mu$.
    Then, we have
    \[
        \E{\mu_X(k^Z - k_{s, X}^Z)}
        = \E{\mu_X(k_X^Z - k_{s, X}^Z)}
        \le \sum_{i>s}\sigma_i,
    \]
    where the expectation is taken regarding the draws of $X$.
\end{prop}

We can also give a bound of the resulting error $\mu(k^Z - k_{s,X}^Z)$
again by using
the arguments from learning theory, but under an additional assumption as stated in the following. 
Nevertheless, Proposition~\ref{prop:emp-eig-decay-x}
is already sufficient for our application in kernel quadrature;
see Theorem~\ref{thm:main-kq-emp}.

\begin{prop}\label{thm:k-s-x}
    Under the same setting as in Proposition~\ref{prop:emp-eig-decay-x},
    if $\ker k(X, Z) \subset \ker k(Z, Z)$ holds almost surely for the draws of $X$,
    we have
    \begin{align*}
        \E{\mu(\sqrt{k^Z - k^Z_{s, X}})}
        \le 2\sqrt{\sum_{i>M}\sigma_i}
        + 4\sqrt{\sum_{i>m}\sigma_i} \icml{ \\
        +\ }{+} \frac{\sqrt{k_{\max}}}{M}\left(\frac{80m^2\log(1+2M)}9 + 69\right)\icml{&}{}.
    \end{align*}
    for any integer $m\ge1$.
\end{prop}

\begin{rem}\label{rem:assumption}
    The assumption $\ker k(X, Z) \subset \ker k(Z, Z)$
    seems to be very hard to check in practice.
    An example with this property is $(\X, k, \mu)$ such that
    $\X = \R^D$ with $D, M > \ell$, the kernel $k$ is just the Euclidean inner product on $\R^D$,
    and $\mu$ is given by a Gaussian distribution with a nonsingular covariance matrix.

    This said, we have some ways to avoid this issue in practice.
    One way is to use $X\cup Z$ instead of $X$ so that the condition automatically holds.
    Then, the above order of estimate should still hold
    when $\ell\ll M$,
    though it complicates the analysis.
    Another way is effective when we use $k_X^Z$ for constructing a kernel quadrature
    from an empirical measure given by $X$ itself; see the next section for details.
\end{rem}

\section{Application to Kernel Quadrature}\label{sec:kernel quadrature}
Let us give error bounds for kernel quadrature
as a consequence of the previous sections.
We are mainly concerned with the kernel quadrature of the form
\eqref{eq:kq-weighted} without weight, i.e., the case when $g=1$
for efficiently discretizing the probability measure $\mu$.

Given an $n$-point quadrature rule $Q_n: f\mapsto \sum_{i=1}^n w_if(x_i)$
with weights $w_i\in\R$ and points $x_i\in\X$,
the worst-case error of $Q_n$ with respect to the RKHS $\H_k$ and
the target measure $\mu$ is defined as
\begin{equation*}
    \wce(Q_n; \H_k, \mu):=\sup_{\lVert f\rVert_{\H_k}\le 1}
    \lvert Q_n(f) - \mu(f)\rvert.
\end{equation*}
Note that it is equal to
$\mmd_k(Q_n, \mu)$,
the maximum mean discrepancy (with $k$)
between $Q_n$ regarded as a (signed) measure
and $\mu$ \citep{gre06}.
We call $Q_n$ {\it convex} if it is a probability measure,
i.e., $w_i\ge0$ and $\sum_{i=1}^nw_i=1$.

Suppose we are given an $s$-rank kernel approximation
$k_\mathrm{app}(x, y)= \sum_{i=1}^s c_i\phi_i(x)\phi_i(y)$ with
$c_i\ge0$ and
$k - k_\mathrm{app}$ being positive definite ($\mu$-almost surely).
The following is taken from \citet[][Theorem 6 \& 8]{hayakawa21b}.

\begin{thm}\label{thm:pkq}
    If an $n$-point convex quadrature $Q_n$
    satisfies
    $Q_n(\phi_i) = \mu(\phi_i)$ for $1\le i\le s$ and
    $Q_n(\sqrt{k - k_\mathrm{app}}) \le \mu(\sqrt{k - k_\mathrm{app}})$,
    then we have
    \[
        \wce(Q_n; \H_k, \mu) \le 2\mu(\sqrt{k - k_\mathrm{app}}).
    \]
    Moreover, such a quadrature $Q_n$ exists with $n=s+1$.
\end{thm}

Although there is a randomized algorithm
for constructing the $Q_n$ stated in the above theorem
\citep[][Algorithm 2 with modification]{hayakawa21b},
it has two issues;
it requires exact values of $\mu(\phi_i)$ (and $\mu(\sqrt{k - k_\mathrm{app}})$)
and its computational complexity has no useful upper bound
unless we have additional assumptions
such as well-behaved moments of test functions $\phi_i$ \citep{hayakawa21a} or structure like a product kernel with a product measure \citep{hayakawa22hyper}.
This said, we can deduce updated convergence results
for outputs of the algorithm as in Remark \ref{rem:result}.

\subsection{Kernel Recombination}
Instead of considering an ``exact" quadrature,
what we do in practice in this low-rank approach is matching the integrals
against a large empirical measure \citep[see also][Section 6]{ada22},
say $\mu_Y=\frac1N\sum_{i=1}^N\delta_{y_i}$ with $Y=(y_i)_{i=1}^N$.
If we have
\begin{equation}
    \begin{cases}
        Q_n(\phi_i)=\mu_Y(\phi_i),\quad 1\le i\le s,\\
        Q_n(\sqrt{k - k_\mathrm{app}})\le \mu_Y(\sqrt{k - k_\mathrm{app}}),
    \end{cases}
    \label{eq:rec-constraint}
\end{equation}
then, from Theorem \ref{thm:pkq} with a target measure $\mu_Y$
and the triangle inequality of MMD,
we have
\begin{align}
    \wce(Q_n; \H_k, \mu)
    &\le \mmd_k(Q_n, \mu_Y) + \mmd_k(\mu_Y, \mu)\nonumber\\
    \icml{}{&}\le 2\mu_Y\icml{&}{}(\sqrt{k-k_\mathrm{app}}) + \mmd_k(\mu_Y, \mu).
    \label{eq:estimate-empirical}
\end{align}
Indeed, such a quadrature $Q_n$ with $n=s+1$
and points given by a subset of $Y$
can be constructed
via an algorithm called {\it recombination} \citep{lit12,tch15,cos20,hayakawa21b}.

Existing approaches of this kernel recombination
have then been using an approximation $k_\mathrm{app}$ typically
given by $k_s^Z$ whose randomness is independent from the sample $Y$,
but it is not a necessary requirement as long as
we can expect an efficient bound of $\mu_Y(\sqrt{k-k_\mathrm{app}})$
in some sense.
Another small but novel observation is that $k-k_\mathrm{app}$
being positive definite is only required on the sample $Y$
in deriving the estimate \eqref{eq:estimate-empirical};
not over the support of $\mu$ in contrast to Theorem~\ref{thm:pkq}.
These observations circumvent the issues mentioned in Remark~\ref{rem:assumption}
when using $k_\mathrm{app}= k_Y^Z$ ($k_{s,X}^Z$ with $X=Y$).

Let us now denote the kernel recombination in a form of function as
$Q_n = \mathrm{KQuad}(k_\mathrm{app}, Y)$,
where the output $Q_n$ is an $n$-point convex quadrature
satisfying $n=s+1$ and \eqref{eq:rec-constraint};
note that the constraint is slightly different from what is given
in \citet[Algorithm 1]{hayakawa21b},
but we can achieve \eqref{eq:rec-constraint}
by replacing $k_{1,\mathrm{diag}}$ with $\sqrt{k_{1,\mathrm{diag}}}$
in the cited algorithm.

We can now prove the performance of low-rank approximations
given in the previous section.
Indeed,  $k_{s,Y}^Z$ and $k_{s, \mu}^Z$
have the following same estimate.

\begin{thm}\label{thm:main-kq-emp}
    Let $Z \subset \X$ be a fixed subset and
    $Y$ be an $N$-point independent sample from $\mu$.
    Then, a random convex quadrature $Q_n = \mathrm{KQuad}(k_{s,Y}^Z, Y)$
    satisfies
    \begin{align}
        &\E{\wce(Q_n; \H_k, \mu)} \icml{\nonumber\\
        &}{}\le 2\mu(\sqrt{k - k^Z})
        + 2\sqrt{\sum_{i>s}\sigma_i}
        + \sqrt{\frac{c_{k,\mu}}N},
        \label{eq:main-kq-emp}
    \end{align}
    where 
    $c_{k,\mu}:=\mu(k)- \iint_{\X\times\X} k(x, y) \dd\mu(x)\dd\mu(y)$
    and the expectation is taken regarding the draws of $Y$.
    The estimate \eqref{eq:main-kq-emp}
    holds also for $Q_n=\mathrm{KQuad}(k_{s,\mu}^Z, Y)$.
\end{thm}

\subsection{Numerical Examples}
In this section, we compare the numerical performance of
$k_{s, Y}^Z$ and $k_{s,\mu}^Z$ for kernel quadrature
with the conventional Nystr{\"o}m approximation
for a non-i.i.d.~$Z$
in the setting that we can explicitly compute the worst-case error.
\paragraph{Periodic Sobolev Spaces.}
The class of RKHS we use is
called periodic Sobolev spaces of functions on $\X = [0, 1]$
(a.k.a. Korobov spaces),
and given by the following kernel
for a positive integer $r$:
\[
    k_r(x, y) = 1 + \frac{(-1)^{r-1}(2\pi)^{2r}}{(2r)!} B_{2r}(\lvert x - y \rvert),
\]
where $B_{2r}$ is the $2r$-th Bernoulli polynomial \citep{wah90,bac17}.
We consider the case $\mu$ being the uniform measure,
where the eigenfunctions of the integral operator $\K$
are known to be $1, \sqrt{2}\cos(2\pi m\, \cdot),
\sqrt{2}\sin(2\pi m\, \cdot)$ with eigenvalues respectively $1, m^{-2r}, m^{-2r}$
for each positive integer $m$.
This RKHS is commonly used for measuring the performance of kernel quadrature
methods \citep{kan16,bac17,bel19,hayakawa21b}.
We also consider its products:
$k_r^{\otimes d}(\bm{x}, \bm{y}) = \prod_{i=1}^d k_r(x_i, y_i)$
and $\mu$ being the uniform measure on the hypercube $\X = [0, 1]^d$.

By considering the eigenvalues,
we can see that
$h_\mu= k_{2r}^{\otimes d}$ for each kernel $k_r^{\otimes d}$
from Remark \ref{rem:square-kernel}.

\paragraph{Experiments.}
In the experiments for the kernel $k_r^{\otimes d}$,
we compared the worst-case error
of $n$-point kernel quadrature rules
given by $Q_n = \mathrm{KQuad}(k_\mathrm{app}, Y)$
with $k_\mathrm{app} = k_s^H, k_s^Z, k_{s, Y}^Z, k_{s,\mu}^Z$
($s=n-1$)
under the following setting:
\begin{itemize}
    \item $Y$ is an $N$-point independent sample from $\mu$
    with $N=n^2$ (Figure~\ref{fig1}) or $N=n^3$ (Figure~\ref{fig2}).
    \item $H$ is the uniform grid $\{i/n\mid i=1, \ldots, n\}$ ($d=1$)
    or the Halton sequence with Owen scrambling \citep{hal60,owe17} ($d\ge2$).
    \item $Z$ is the union of $H$ and another $20n$-point independent sample
    from $\nu^{\otimes d}$, where $\nu$ is the $1$-dimensional $(2, 5)$-Beta distribution,
    whose density is proportional to $x(1-x)^4$ for $x\in[0, 1]$.
\end{itemize}

We additionally compared `{\bf Monte Carlo}':
uniform weights $1/n$ with
i.i.d.~sample $(x_i)_{i=1}^n$ from $\mu$,
`{\bf Uniform Grid}' ($d=1$):
points in $H$ with uniform weights $1/n$
(known to be {\it optimal} for each $n$),
and 
`{\bf Halton}' ($d\ge2$):
points in an independent copy of $H$ with uniform weights $1/n$.

The aim of this experiment was to see if
the proposed methods ($k_{s,Y}^Z$ and $k_{s,\mu}^Z$)
can actually recover a `good' subspace of the RKHS given by $k^Z$
with $Z$ not summarizing $\mu$.
To do so, we mixed $H$ (a `good' summary of $\mu$) and an i.i.d.~sample
from $\nu$ to determine $Z$.

Figure~\ref{fig1} shows the results
for $(d, r)=(1, 1), (2, 1), (3, 3)$
with $N=n^2$ and $n=4, 8, 16, 32, 64, 128$.
From Figure~\ref{fig1}(a, b),
we can see that our methods indeed recover (and perform slightly better than)
the rate of $k^H$ from a contaminated sample $Z$.
In Figure~\ref{fig1}(c), the four
low-rank methods all perform equally well,
and it seems that the dominating error is given by the term caused by $\mmd_k(\mu_Y, \mu)$.

\begin{figure}
    \vskip 0.2in
    \centering
    \includegraphics[width=0.9\hsize]{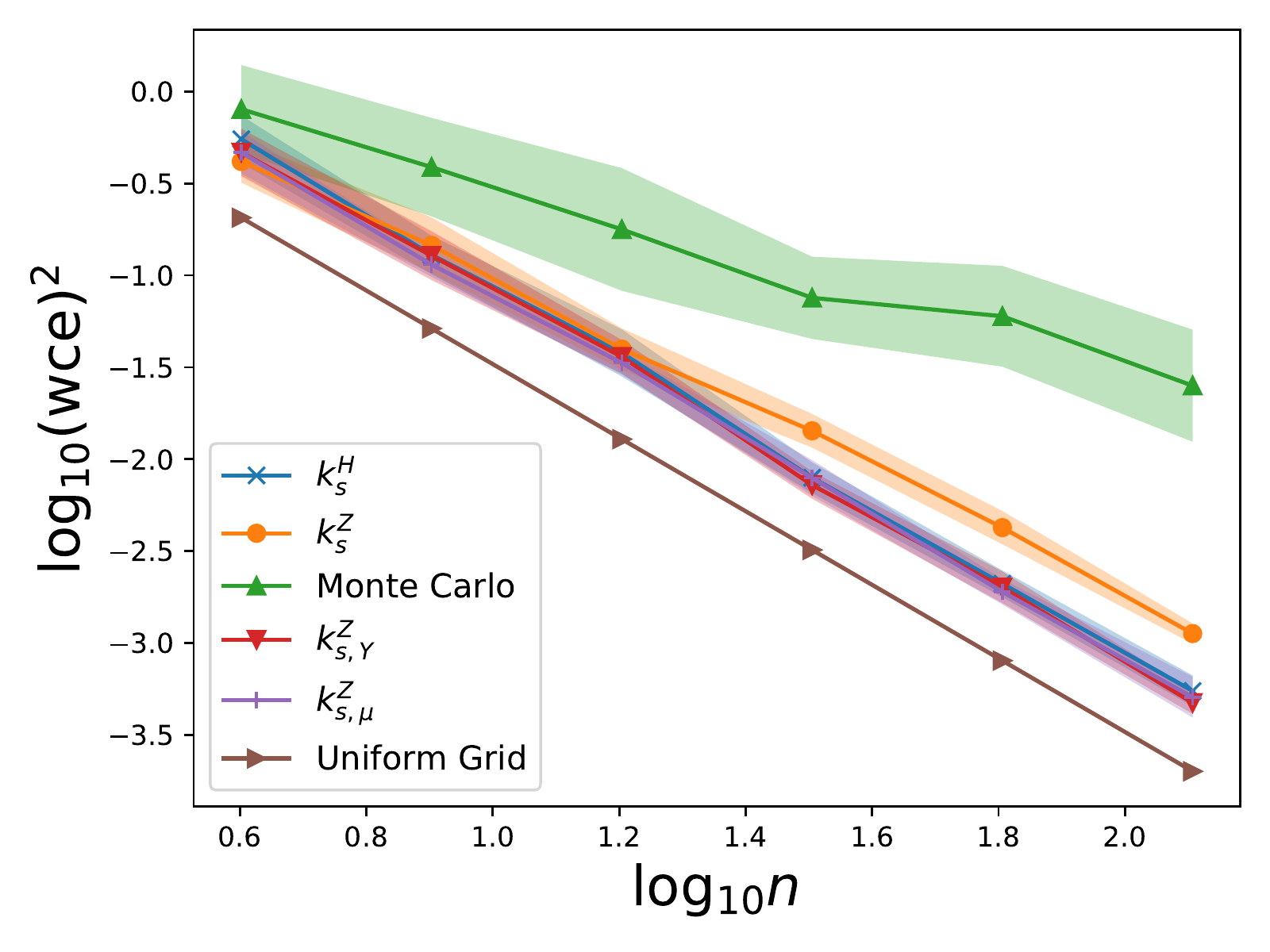}
    \vskip -0.1in
    \caption*{(a) $d=1$, $r=1$}
    \includegraphics[width=0.9\hsize]{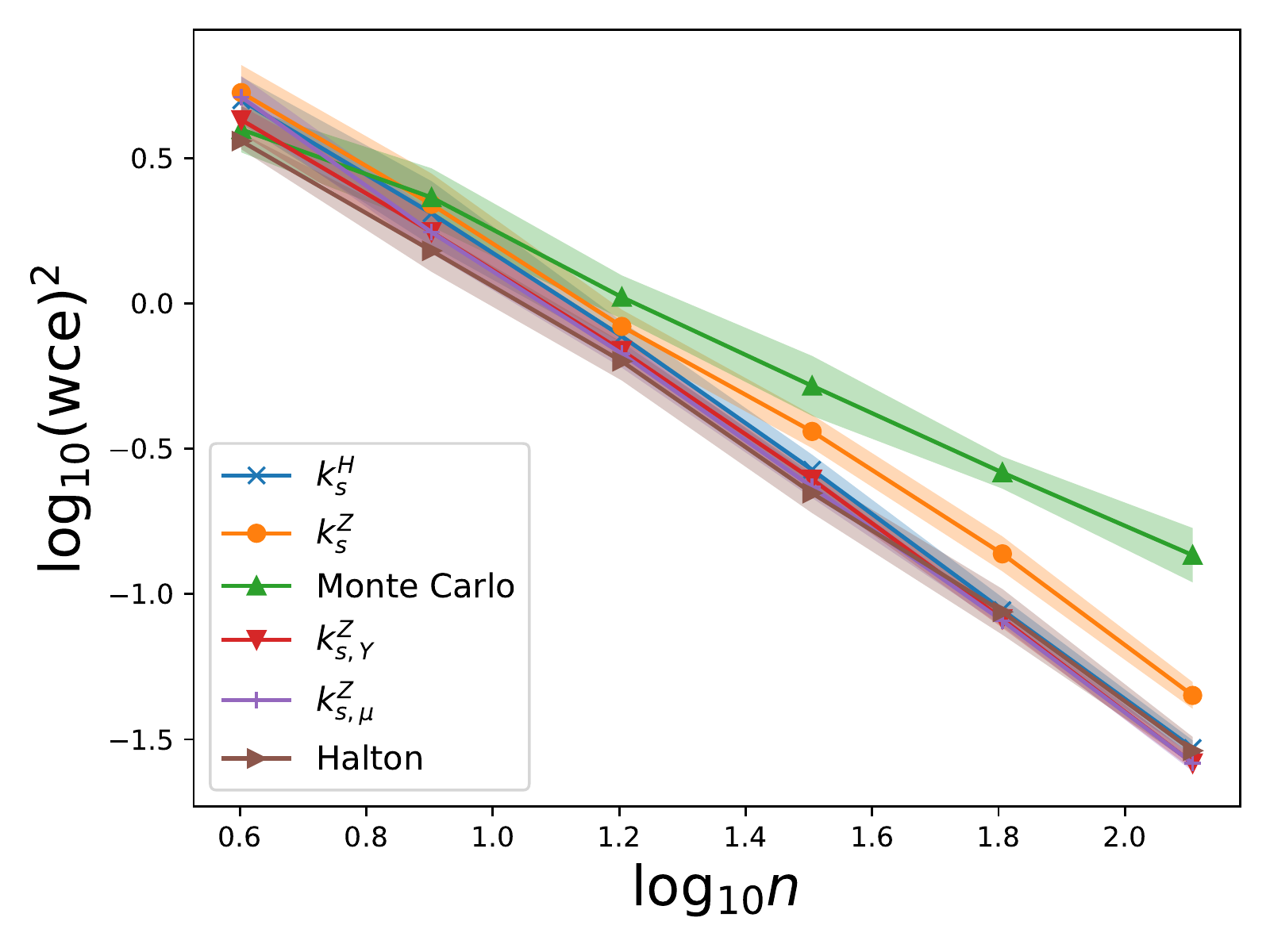}
    \vskip -0.1in
    \caption*{(b) $d=2$, $r=1$}
    \includegraphics[width=0.9\hsize]{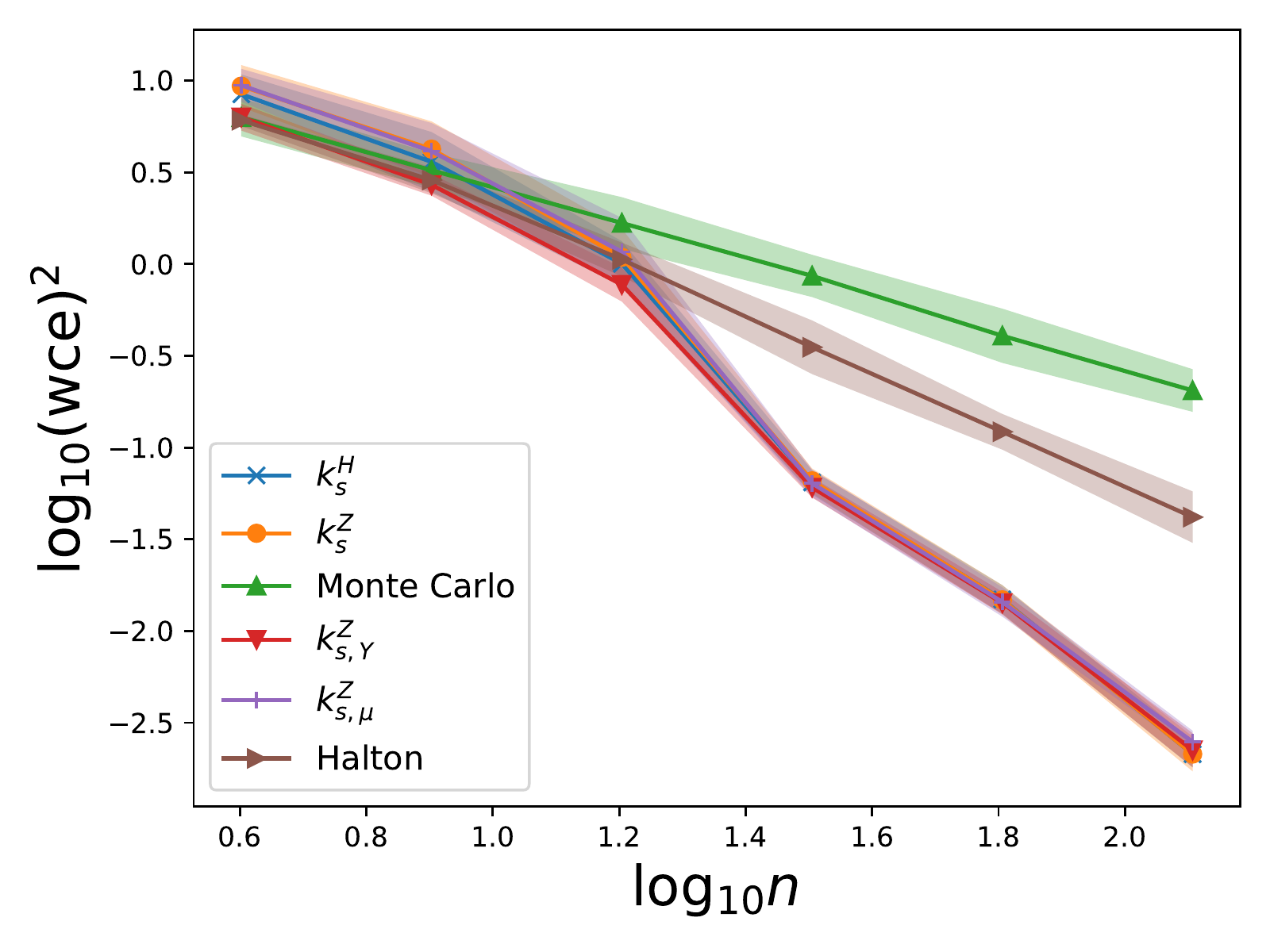}
    \vskip -0.1in
    \caption*{(c) $d=3$, $r=3$}
    \caption{Experiments in periodic Sobolev spaces
    with reproducing kernel $k_r^{\otimes d}$.
    Average of $\log_{10}(\wce(Q_n;\H_k, \mu)^2)$ over 20 samples plotted with their standard deviation.}
    \label{fig1}
    \vskip -0.2in
\end{figure}

\begin{figure}[ht]
    \vskip 0.2in
    \centering
    \includegraphics[width=0.9\hsize]{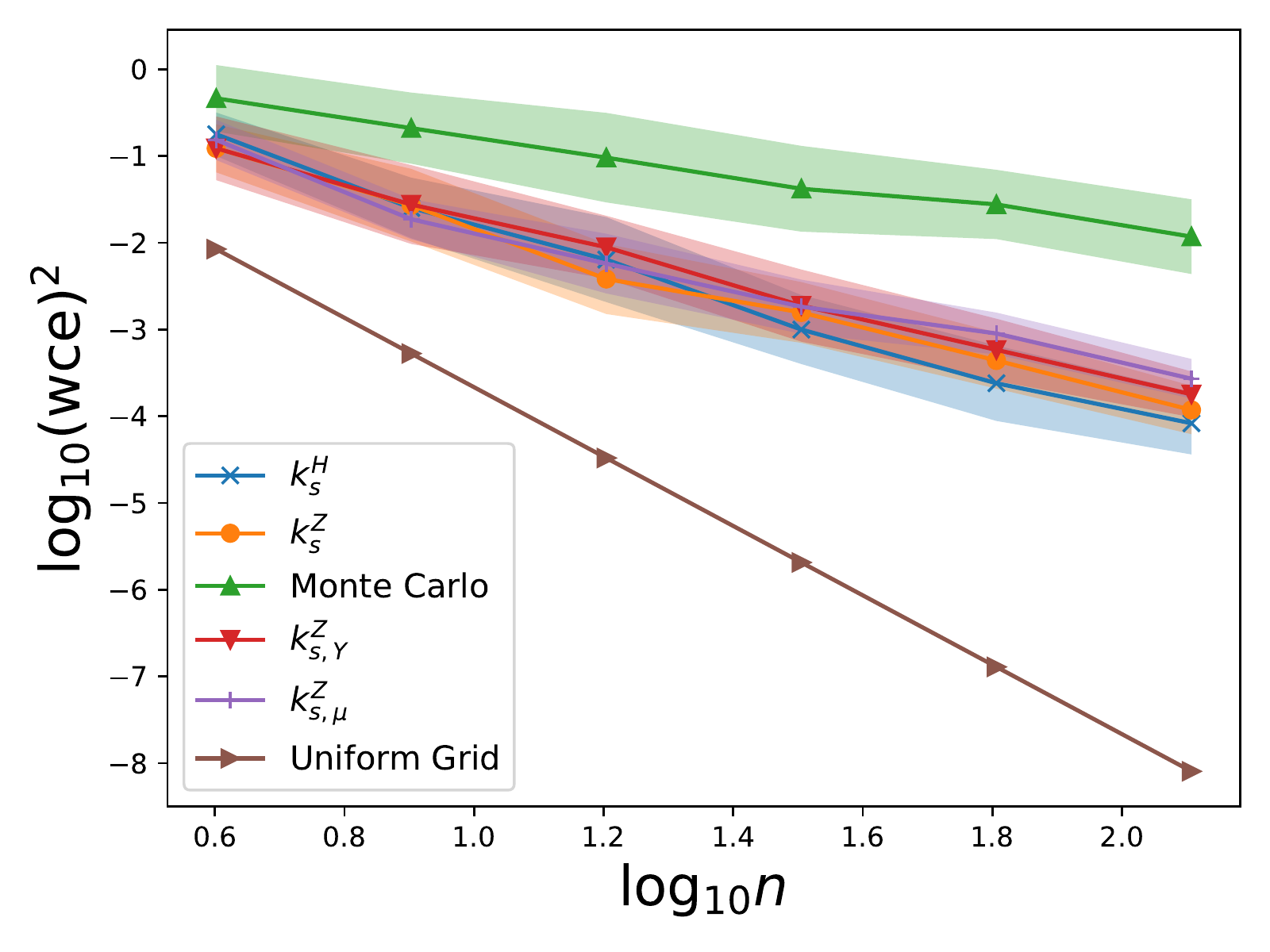}
    \vskip -0.1in
    \caption*{(a) $N = n^2$}
    \includegraphics[width=0.9\hsize]{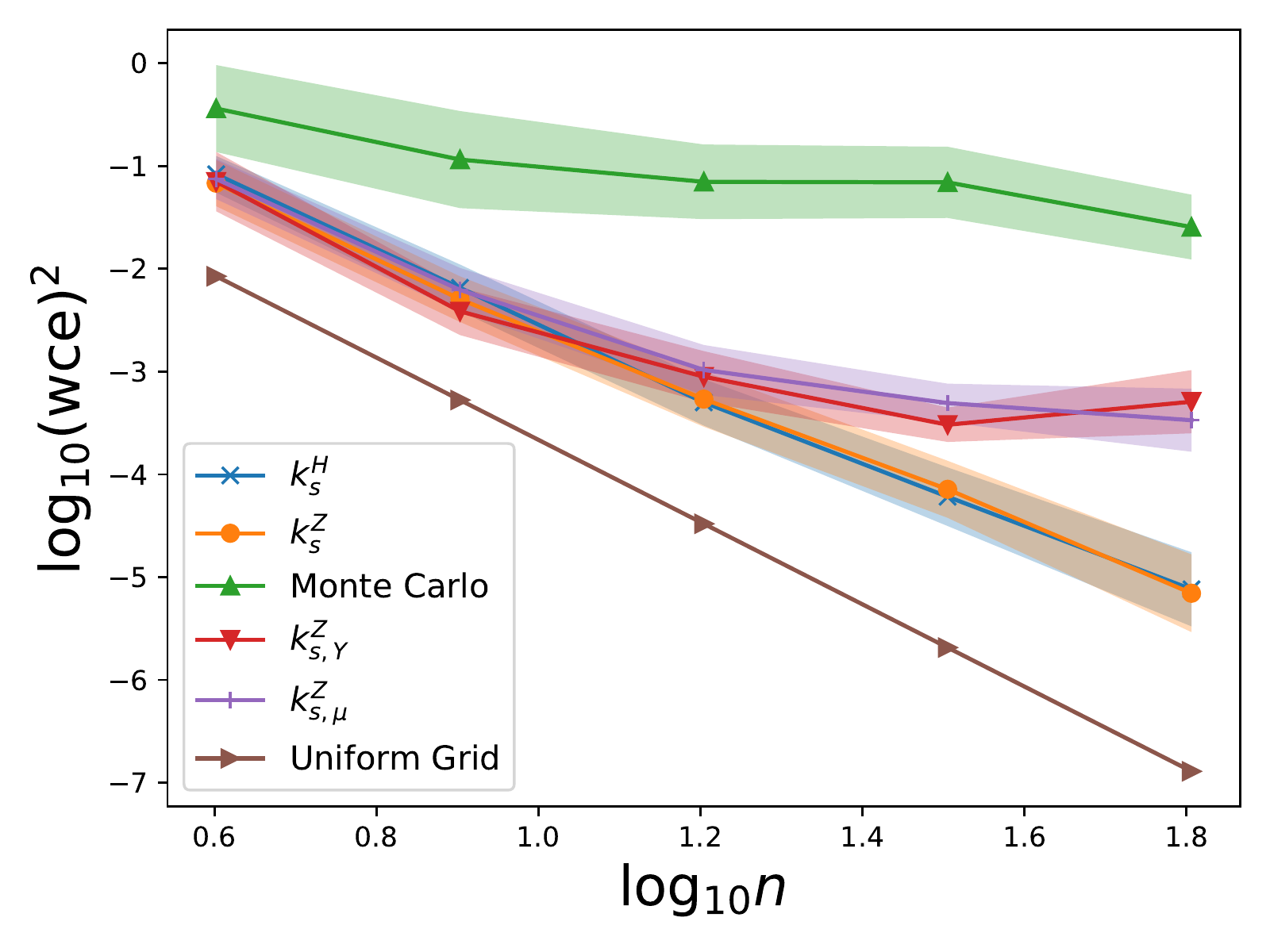}
    \vskip -0.1in
    \caption*{(b) $N = n^3$}
    \caption{Experiments in $k_2$
    with $N = n^2, n^3$ for recombination algorithms.
    Average of $\log_{10}(\wce(Q_n;\H_k, \mu)^2)$ over 20 samples plotted with their standard deviation.}
    \label{fig2}
    \vskip -0.2in
\end{figure}

Figure~\ref{fig2} shows the results for $(d, r)=(1, 2)$
with $N = n^2$ or $N = n^3$ and $n=4,8,16,32,64$.
In this case, we can see that $k_{s,Y}^Z$ or $k_{s,\mu}^Z$
eventually suffers from the numerical instability,
which is also reported by \citet{san16}.
Since their error inflation is not completely
hidden even in the case $N=n^2$
unlike the previous experiments,
one possible reason for the instability
is that taking the pseudo-inverse
of $k(Z, Z)$ or $h_\mu(Z, Z)^{1/2}$ in the algorithm
becomes highly unstable when the spectral decay is fast.
Although they have preferable guarantees in theory,
its numerical error seems to harm the overall efficiency,
and this issue needs to be addressed e.g. by circumventing
the use of pseudo-inverse in future work.

\begin{rem}\label{rem:k-s-z-stable}
    Unlike the kernel quadrature with $k_{s,\mu}^Z$ or $k_{s,Y}^Z$,
    that with $k_s^Z$ does not suffer from
    a similar numerical instability despite the use of $k(Z, Z)^+_s$.
    This phenomenon can be explained by the nature of 
    \citet[][Algorithm 1]{hayakawa21b};
    it only requires (stable)
    {\it test functions} $\phi_i = u_i^\top k(Z, \cdot)$
    ($i=1,\ldots,s$) for its equality constraints,
    where $u_i$ is the $i$-th eigenvector of $k(Z, Z)$,
    while the (possibly unstable)
    diagonal term $k_s^Z(x, x)$ appears
    in the inequality constraint,
    which can empirically be omitted \citep[][Section E.2]{hayakawa21b}.
\end{rem}

\paragraph{Computational Complexity.}
By letting $\ell, N$ (larger than $s$)
respectively be the cardinality of $Z$ and $Y$,
we can express the computational steps
of $\mathrm{KQuad}(k_\mathrm{app}, Y)$
with $k_\mathrm{app} = k_s^Z, k_{s,Y}^Z, k_{s,\mu}^Z$
as follows:
\begin{itemize}
    \item Using $k_s^Z$ takes $\ord{s\ell N + s\ell^2 + s^3\log(N/s)}$,
    but it can be reduced to $\ord{\ell N + s\ell^2 + s^3\log(N/s)}$
    by omitting the (empirically unnecessary) inequality constraint
    \citep[][Remark 2]{hayakawa21b}.
    \item Using $k_{s,Y}^Z$ takes
    $\ord{\ell^3 + \ell^2 N + s^3\log(N/s)}$,
    where $\ord{\ell^3}$ and $\ord{\ell^2 N}$ respectively
    come from computing $k(Z, Z)^+$ and $h_Y(Z, Z)$.
    \item Using $k_{s,\mu}^Z$ takes
    $\ord{\ell^3 + s\ell N + s^3\log(N/s)}$
    (if $h_\mu$ available), where $\ord{\ell^3}$
    is from computing $k(Z, Z)^+$.
\end{itemize}
For example, in the case of Figure~\ref{fig1}(c) with $n=128$,
the average time per one trial was respectively
26.5, 226, 216 seconds for $k_s^Z, k_{s,Y}^Z, k_{s,\mu}^Z$,
while it was 52.6, 57.8, 41.2 seconds for the case
of Figure~\ref{fig2}(b) with $n=64$.\footnote{
All the experiments were conducted on a MacBook Pro with Apple M1 Max chip and 32GB unified memory.
Code is available at the {\ttfamily nystrom} folder in \url{https://github.com/satoshi-hayakawa/kernel-quadrature}.}

\section{Concluding Remarks}
In this paper,
we have studied the performance
of several Nystr{\"o}m-type approximations $k_\mathrm{app}$
of a positive definite kernel $k$ associated with
a probability measure $\mu$,
in terms of the error $\mu(\sqrt{k - k_\mathrm{app}})$.
We first improved the bounds for $k_s^Z$,
the conventional Nystr{\"o}m approximation based on
an i.i.d.~$Z$ and the use of SVD, by leveraging  results in statistical learning theory.
We then went beyond the i.i.d.~setting and considered general $Z$ including DPPs; we further introduced two competitors of $k_s^Z$,
i.e., $k_{s,\mu}^Z$ and $k_{s,X}^Z$,
which are given by directly computing the Mercer decomposition
of the finite-rank kernel $k^Z$ against
the measure $\mu$ and the empirical measure $\mu_X$,
respectively.
Finally, we used our results to improve
the theoretical guarantees for convex kernel quadrature \citet{hayakawa21b},
and provided numerical results to illustrate the difference
between the conventional $k_s^Z$ and
the newly proposed $k_{s,\mu}^Z$ and $k_{s,X}^Z$.

Despite its nice theoretical properties,
a limitation of our second contribution,
i.e., the proposed kernel approximations, is that
they involve the computation of a pseudo-inverse,
which can be numerically unstable
when there is a rapid spectral decay.
This point should be addressed in future work,
but one promising approach in the context of kernel quadrature
is to conceptually learn from the stability of $k_s^Z$
mentioned in Remark~\ref{rem:k-s-z-stable};
if we see the construction of the low-rank kernel
as optimization of the vectors $u_i$ for which
functions $u_i^\top k(Z, \cdot)$ well approximate
$\H_{k^Z}$ in terms of $L^2(\mu)$ metric,
we can possibly leverage the stability of convex optimization
for instance.

\section*{Acknowledgements}
The authors would like to thank Ken'ichiro Tanaka and the anonymous reviewers for helpful comments.
This work was supported in part by the EPSRC [grant number EP/S026347/1], in part by The Alan Turing Institute under the EPSRC grant EP/N510129/1, the Data Centric Engineering Programme (under the Lloyd’s Register Foundation grant G0095), the Defence and Security Programme (funded by the UK Government) and the Office for National Statistics \& The Alan Turing Institute (strategic partnership) and in part by the Hong Kong Innovation and Technology Commission (InnoHK Project CIMDA).

\bibliography{cite}% BibTeX を使う場合
\bibliographystyle{icml2023}

%%%%%%%%%%%%%%%%%%%%%%%%%%%%%%%%%%%%%%%%%%%%%%%%%%%%%%%%%%%%

\newpage
\appendix
\onecolumn

\section{Tools from statistical learning theory}\label{sec:stat-learn}
In this section, $\F$ always denotes a class of functions from $\X$ to $\R$,
i.e., $\F\subset\R^\X$.
Let us define the Rademacher complexity of $\F$
with respect to the sample $Z = (z_i)_{i=1}^\ell\subset\X$
as follows \citep[e.g., ][Definition 3.1]{moh18}:
\[
    \mathcal{R}_Z(\F):=\E{\sup_{f\in\F}
        \frac1\ell\sum_{j=1}^\ell s_jf(z_j)
    \lmid Z}
\]
where the conditional expectation is taken with regard to
the Rademacher variables,
i.e., i.i.d. variables $s_j$ uniform in $\{\pm1\}$.

The following is a version of the uniform law of large numbers,
though we only use the one side of the inequality.
\begin{prop}[{\citealp[][Theorem 3.3]{moh18}}]\label{prop:ulln}
    Let $Z$ be an $\ell$-point independent sample from $\mu$.
    If there is a $B>0$ such that $\lVert f\rVert_\infty\le B$ for every $f\in\F$,
    then with probability at least $1-\delta$, we have
    \[
        \sup_{f\in\F}(\mu(f) - \mu_Z(f)) \le 2\E{\mathcal{R}_Z(\F)}
        + \sqrt{\frac{2B^2}\ell \log\frac1\delta}.
    \]
\end{prop}

For a pseudo metric $d$ on $\F$,
we denote the {\it $\ve$-convering number} of $\F$ by $\N(\F, d; \ve)$.
Namely, $\N(\F, d; \ve)$ is the infimum of positive integers $N$ such that
there exist $f_1, \ldots, f_N\in\F$ satisfying
$\min_{1\le i\le N}d(f_i, g)\le\ve$ for all $g\in\F$.

Let us define a pseudo-metric $d_Z(f, g):=\sqrt{\frac1\ell\sum_{j=1}^\ell (f(z_i) - g(z_i))^2}$.
The following assertion is a version of Dudley's integral entropy bound
\citep[][Lemma A.3; see \citet{sre10+} for a correction of the constant]{sre10}.
\begin{prop}[Dudley integral]\label{prop:dudley}
For any $\ell$-point sample $Z=(z_i)_{i=1}^\ell \subset\X$,
we have
    \[
        \mathcal{R}_Z(\mathcal{F})
        \le \frac{12}{\sqrt{\ell}}\int_0^\infty \sqrt{\log \N(\mathcal{F}, d_Z; \ve)}\dd\ve.
    \]
\end{prop}

The following is a straightforward modification of \citet[][Lemma 4]{sch20}
tailored to our setting.
It originates from an analysis of empirical risk minimizers,
and this kind of technique has also been known in earlier work under the name of local Rademacher complexities
\citep{gyorfi2006distribution,koltchinskii2006local,gine2006concentration}.
\begin{prop}\label{prop:sh}
    Let $\F\subset L^\infty(\mu)$ be a set of functions
    with $f\ge0$ and $\lVert f\rVert_{L^\infty(\mu)}\le F$ for all $f\in\F$,
    where $F>0$ is a constant.
    If $\hat{f}$ is a random function in $\F$ possibly depending on $Z$,
    then, for every $\ve>0$, we have
    \[
        \E{\mu(\hat{f})} \le 2\E{\mu_Z(\hat{f})} + \frac{F}\ell\left(
            \frac{80}9\log N + 64
        \right) + 5\ve,
    \]
    where $N:=\max\{3, \N(\F, \lVert\cdot\rVert_{L^1(\mu)}; \ve)\}$.
\end{prop}

\begin{proof}
    The proof here essentially follows the original proof,
    where we re-compute the constants as the condition is slightly different;
    see also \citet[][Theorem 2.6]{hayakawa-suzuki} and its remark.
     
    Let $Z^\prime = (z_1^\prime, \ldots z_\ell^\prime)$ be an independent copy of $Z$.
    Let $\F_\ve$ be an $\ve$-covering of $\F$ in $L^1(\mu)$
    with the cardinality $N$
    and $f^*$ be a random element of $\F_\ve$ such that
    $\mu(\lvert \hat{f}-f^*\rvert)\le \ve$.
    Then, we have
    \begin{equation}
        \left\lvert\E{\mu_Z(\hat{f})} - \E{\mu(\hat{f})}\right\rvert
        =\left\lvert\E{\frac1\ell\sum_{i=1}^\ell(\hat{f}(z_i)-\hat{f}(z_i^\prime))}\right\rvert
        \le \E{\left\lvert\frac1n\sum_{i=1}^\ell(f^*(z_i) - f^*(z_i^\prime))\right\rvert}
        +2\ve
        \label{eq:sh1}
    \end{equation}
    
    Define $T:= \max_{f\in\F_\ve}\sum_{i=1}^\ell(f(z_i) - f(z_i^\prime))/r(f)$,
    where we let $r(f):=\max\{c\sqrt{\ell^{-1}\log N}, \sqrt{\mu(f)}\}$
    for each $f\in\F_\ve$ with a constant $c>0$ fixed afterwards. Thus, we obtain
    \begin{equation}
        \E{\left\lvert\frac1n\sum_{i=1}^\ell(f^*(z_i) - f^*(z_i^\prime))\right\rvert}
        \le \E{\frac{r(f^*)T}\ell}
        \le \frac12\E{r(f^*)^2} + \frac1{2\ell^2}\E{T^2}.
        \label{eq:sh2}
    \end{equation}
    
    The first term can evaluated as
    \begin{equation}
        \E{r(f^*)^2}\le c^2\frac{\log N}\ell + \E{\mu(f^*)}
        \le c^2\frac{\log N}\ell + \mathbb{E}[\mu(\hat{f})] + \ve.
        \label{eq:sh3}
    \end{equation}
    For the second term,
    we first have
    \[
        \sum_{i=1}^\ell\E{\left(\frac{f(z_i) - f(z_i^\prime)}{r(f)}\right)^2}
        \le
        \sum_{i=1}^\ell\E{\frac{f(z_i)^2 + f(z_i^\prime)^2}{r(f)^2}}
        \le 2F\ell,
        \qquad f\in\F_\ve.
    \]
    Since
    we have $\lvert f(z_i) - f(z_i^\prime)\rvert/r(f)\le 2F/r(f)
    \le 2F\frac{\sqrt{\ell}}{c\sqrt{\log N}}$ uniformly for $f\in\F_\ve$,
    Bernstein's inequality combined with the union bound yields
    \[
        \P{T^2\ge t}=\P{T\ge\sqrt{t}}
        \le 2N\exp\left(-\frac{t}{4F(\ell + \frac{\sqrt{\ell t}}{3c\sqrt{\log N}})}\right)
        \le 2N\exp\left(-\frac{3c\sqrt{\log N}}{8F\sqrt{\ell}}\sqrt{t}\right)
    \]
    for $t\ge 9c^2\ell\log N$.
    Therefore, we have
    \begin{align*}
        \E{T^2}
        = \int_0^\infty \P{T^2\ge t} \dd t
        &\le 9c^2\ell\log N + \int_{9c^2\ell\log N}^\infty2N\exp\left(-\frac{3c\sqrt{\log N}}{8F\sqrt{\ell}}\sqrt{t}\right)\dd t\\
        &= 9c^2\ell\log N + 4N\left(
            8F\ell + \frac{64F^2\ell}{9c^2\log N}
        \right)\exp\left(
            - \frac{9c^2\log N}{8F}
        \right)
    \end{align*}
    Let us now set $c=\sqrt{8F/9}$ so that $9c^2=8F$.
    Then, we obtain $\E{T^2}\le 8F\ell\log N + 64F\ell$
    since $N\ge3$ by assumption.
    By combining it with \eqref{eq:sh1}--\eqref{eq:sh3},
    we finally obtain
    \[
        \left\lvert\E{\mu_Z(\hat{f})} - \E{\mu(\hat{f})}\right\rvert
        \le \frac12\E{\mu(\hat{f})} + \frac{(\frac{40}9 F\log N + 32F)}\ell + \frac52\ve,
    \]
    from which the desired inequality readily follows.
\end{proof}

\section{Proofs}\label{sec:proofs}
\subsection{Properties of the pseudo-inverse}
For a matrix $A\in\R^{m\times n}$,
its Moore--Penrose pseudo-inverse $A^+$
\citep{pen55}
is defined as the unique matrix $X\in\R^{n\times m}$
that satisfies
\[
    AXA=A,
    \quad XAX=X,
    \quad (AX)^\top = AX,
    \quad (XA)^\top = XA.
\]
It also satisfies that
$A^+ A$ is the orthogonal projection onto
the orthogonal complement
of $\ker A$ (the range of $A^\top$),
while $AA^+$ is the orthogonal projection onto the range of $A$
\citep{pen55,shi72}.
We use these general properties of $A^+$
throughout Section~\ref{sec:proofs}.
See e.g. \citet{dri05}
for the concrete construction of such a matrix.

\subsection{Proof of Lemma \ref{lem:proj-trivial}}
\begin{proof}
    Recall that
    we have the SVD $k(Z, Z) = U\mathop\mathrm{diag}(\lambda_1, \ldots,\lambda_\ell)U^\top$
    with an orthogonal matrix $U = [u_1, \ldots, u_\ell]$.
    and $\lambda_1\ge\cdots\ge\lambda_\ell\ge0$.
    By using this notation, we have
    \begin{equation}
        k_s^Z(x, y)
        = \sum_{\substack{1\le j\le s\\ \lambda_j>0}}\frac1\lambda_j
        (u_j^\top k(Z, x))(u_j^\top k(Z, y)).
        \label{eq:explicit-k-z}
    \end{equation}
    If we denote by $Q_j:\H_k\to\H_k$
    the projection onto $\mathop\mathrm{span}\{u_j^\top k(Z, \cdot)\}$,
    we have
    \begin{align}
        (u_j^\top k(Z, x))(u_j^\top k(Z, y))
        &= \ip{u_j^\top k(Z, \cdot), k(\cdot, x)}_{\H_k} \ip{u_j^\top k(Z, \cdot), k(\cdot, y)}_{\H_k}
        \nonumber\\
        &= \lVert u_j^\top k(Z, \cdot) \rVert_{\H_k}^2
        \ip{Q_jk(\cdot, x), Q_jk(\cdot, y)}_{\H_k}\nonumber\\
        &= \lambda_j \ip{Q_jk(\cdot, x), Q_jk(\cdot, y)}_{\H_k},
        \label{eq:emp-ip}
    \end{align}
    where the last inequality follows from $\ip{u_i^\top k(Z, \cdot), u_j^\top k(Z, \cdot)}_{\H_k}
    =u_i^\top k(Z, Z)u_j = \delta_{ij}\lambda_j$.
    Now let $\tilde{P}_{Z,s}$ be the orthogonal projection onto $\mathop\mathrm{span}\{u_j^\top k(Z, \cdot)\}_{j=1}^s$ in $\H_k$.
    We prove $\tilde{P}_{Z, s} = P_{Z, s}$.
    From the orthogonality of $\{u_j^\top k(Z, \cdot)\}_{j=1}^s$ we have
    $\tilde{P}_{Z, s} = \sum_{j=1}^sQ_j$ and
    \begin{align*}
        \ip{k(\cdot, x), k_s^Z(\cdot, y)}_{\H_k}= 
        k_s^Z(x, y)
        &= \sum_{j=1}^s \ip{Q_jk(\cdot, x), Q_jk(\cdot, y)}_{\H_k}\\
        &= \ip{\tilde{P}_{Z, s}k(\cdot, x), \tilde{P}_{Z, s}k(\cdot, y)}_{\H_k}
        = \ip{k(\cdot, x), \tilde{P}_{Z, s}k(\cdot, y)}_{\H_k}
    \end{align*}
    for all $x,y\in\X$.
    In particular, $k_s^Z(\cdot, y) = \tilde{P}_{Z, s}k(\cdot, y)$,
    so we have $\tilde{P}_{Z, s} = P_{Z, s}$.
\end{proof}

\subsection{Proof of Lemma \ref{lem:emp}}
\begin{proof}
    The inequality follows from Cauchy--Schwarz.
    Let us prove the equality.

    We use the notation $Q_j$ from the proof of Lemma \ref{lem:proj-trivial}.
    We first obtain $P_Z k(\cdot, z_i) = k(\cdot, z_i)$ for $i=1,\ldots,\ell$,
    since $P_Z$ is a projection onto $\mathop\mathrm{span}\{k(\cdot, z_i)\}_{i=1}^\ell$.
    Thus, we have $P_{Z, s}^\perp k(\cdot, z_i) = (P_Z - P_{Z, s}) k(\cdot, z_i)
    = (Q_{s+1}+\cdots+Q_\ell) k(\cdot, z_i)$,
    and so
    \begin{align*}
        \frac1\ell\sum_{i=1}^\ell
        \lVert P_{Z,s}^\perp k(\cdot, z_i)\rVert_{\H_k}^2
        =\frac1\ell\sum_{i=1}^\ell\sum_{\substack{s+1\le j\le \ell\\ \lambda_j>0}}
            \frac1{\lambda_j}
            (u_j^\top k(Z, z_i))^2
    \end{align*}
    by using \eqref{eq:emp-ip}.
    Since $k(Z, Z)=U\mathop\mathrm{diag}(\lambda_1, \ldots, \lambda_\ell)U^\top
    = \sum_{i=1}^\ell \lambda_i u_i u_i^\top$,
    we can explicitly calculate
    \[
        u_j^\top k(Z, z_i) = u_j^\top \sum_{i=1}^\ell \lambda_i u_i u_i^\top \bm{1}_j
        = \lambda_j u_i^\top\bm{1}_j,
    \]
    where $\bm{1}_j \in\R^\ell$ is the vector with $1$ in the $j$-th coordinate and
    $0$ in the other coordinates.
    As $U$ is an $\ell\times\ell$ orthogonal matrix,
    we actually have $\sum_{i=1}^\ell (u_i^\top\bm{1}_j)^2 = 1$ for each $j=1,\ldots,\ell$.
    \begin{equation}
        \frac1\ell\sum_{i=1}^\ell\sum_{\substack{s+1\le j\le \ell\\ \lambda_j>0}}
            \frac1{\lambda_j}
            (u_j^\top k(Z, z_i))^2
        = \frac1\ell \sum_{i=1}^\ell\sum_{j=s+1}^\ell
        \lambda_j (u_i^\top \bm{1}_j)^2
        = \frac1\ell\sum_{j=s+1}^\ell \lambda_j,
        \label{eq:empirical-trace}
    \end{equation}
    and the proof is complete.
\end{proof}

\subsection{Proof of Lemma \ref{lem:trace}}
\begin{proof}
    From the min-max principle,
    we have
    \begin{equation}
        \lambda_j = \min_{\substack{V_{j-1}\subset \R^\ell \\ \dim V_{j-1} \le j-1}}
        \max_{x_j \in V_{j-1}^\perp,\ \lVert x_j\rVert_2 = 1} x_j^\top k(Z, Z)x_j,
        \label{eq:min-max-mat}
    \end{equation}
    where $V_{j-1}$ is a linear subspace of $\R^\ell$.
    Recall the Mercer expansion $k(x, y) = \sum_{i=1}^\infty \sigma_ie_i(x)e_i(y)$.
    By letting $e_j(Z) = (e_j(z_1), \ldots, e_j(z_\ell))^\top\in\R^\ell$,
    we can write
    $k(Z, Z) = \sum_{i=1}^\infty \sigma_ie_i(Z)e_i(Z)^\top$.
    %almost surely, as $Z$ is an $\ell$-point sample from $\mu$.
    We assume that this equality holds in the following.
    We especially write the remainder term as
    $k_{s+1}(Z, Z):=k(Z, Z) - \sum_{i=1}^s \sigma_ie_i(Z)e_i(Z)^\top$
    
    Consider taking $V_s = \mathop\mathrm{span}\{e_1(Z), \ldots, e_s(Z)\}$
    and
    \[
        x_j\in\mathop\mathrm{argmax}_{x \in V_{j-1}^\perp,\ \lVert x\rVert_2 = 1}
        x^\top k(Z, Z)x,
        \qquad
        V_j = \mathop\mathrm{span}(V_{j-1} \cup \{x_j\})
    \]
    for $j=s+1, \ldots, \ell$
    in \eqref{eq:min-max-mat}.
    Then, $\lambda_j^\prime := x_j^\top k(Z, Z)x$ satisfies $\lambda_j\le\lambda_j^\prime$,
    and so we have
    \[
        \sum_{j=s+1}^\ell \lambda_j
        \le \sum_{j=s+1}^\ell \lambda_k^\prime
        = \sum_{j=s+1}^\ell x_j^\top k(Z, Z) x_j
        = \sum_{j=s+1}^\ell x_j^\top k_{s+1}(Z, Z) x_j,
    \]
    where we have used that $x_j^\top e_i(Z) = 0$ for any $i\le s < j$
    in the last inequality.
    By taking some $\{x_1 ,\ldots, x_s\}\subset\R^\ell$,
    we can make $\{x_1, \ldots, x_\ell\}$ a orthonormal basis of $\R^\ell$,
    so we obtain
    \[
        \sum_{j=s+1}^\ell \lambda_j
        \le \sum_{j=s+1}^\ell x_j^\top k_{s+1}(Z, Z) x_j
        \le \sum_{j=1}^\ell x_j^\top k_{s+1}(Z, Z) x_j
        = \mathop\mathrm{tr} k_{s+1}(Z, Z).
    \]
    
    Therefore, we have
    \[
        \frac1\ell\sum_{j=s+1}^\ell \lambda_j
        \le \frac1\ell \mathop\mathrm{tr} k_{s+1}(Z, Z)
        = \frac1\ell \sum_{i=1}^\ell k_{s+1}(z_i, z_i),
    \]
    and we obtain the desired inequality in expectation since
    $\E{k_{s+1}(z_i, z_i)}=\sum_{j=s+1}^\infty\sigma_j$.
\end{proof}

\subsection{Proof of Theorem \ref{thm:low-dim-proj}}
We first prove the following generic proposition by exploiting
the ingredients given in Section \ref{sec:stat-learn}.

\begin{prop}\label{prop:aux}
    Let $Q$ be an arbitrary deterministic $m$-dimensional orthogonal projection in $\H_k$.
    Then, for any random orthogonal projection $P$ possibly depending on $Z$,
    we have
    \begin{equation}
        \mu(\lVert PQk(\cdot, x)\rVert_{\Hil_k}) \le
        \mu_Z(\lVert PQk(\cdot, x)\rVert_{\Hil_k})
        + \sqrt{\frac{k_{\max}}\ell}\left(36m + \sqrt{2\log\frac1\delta}\right)
        \label{eq:main-root-hp}
    \end{equation}
    with probability at least $1-\delta$.
    
    Furthermore, with regard to the expectation, we also have
    \begin{equation}
        \E{\mu(\lVert PQk(\cdot, x)\rVert_{\Hil_k})}
        \le 2\E{\mu_Z(\lVert PQk(\cdot, x)\rVert_{\Hil_k})}
        + \frac{\sqrt{k_{\max}}}{\ell}\left(\frac{80m^2\log(1+2\ell)}9 + 69\right).
        \label{eq:main-root-ex}
    \end{equation}
\end{prop}
\begin{proof}
    Let $\{v_1, \ldots, v_m\}$ be an orthonormal basis of $Q\H_k$.
    Let also $\{u_i\}_{i\in I}$ and
    $\{u_i\}_{i\in J}$ be respectively
    an orthonormal basis of $P\H_k$ and $(P\H_k)^\perp$,
    so $\{u_i\}_{i\in I\cup J}$ is an orthonormal basis of $\H_k$.
    
    Let us compute $\lVert PQk(\cdot, x)\rVert_{\H_k}^2$.
    Since we have
    \[
        PQk(\cdot, x) = P\left(
            \sum_{j=1}^m \ip{v_j, k(\cdot, x)}_{\H_k} v_j
        \right)
        = \sum_{j=1}^m v_j(x) Pv_j
        = \sum_{i\in I}\sum_{j=1}^m v_j(x) \ip{u_i, v_j}_{\H_k} u_i
    \]
    (where we can exchange the summation as they converge in $\H_k$),
    we obtain
    \[
        \lVert PQk(\cdot, x)\rVert_{\H_k}^2
        = \sum_{i\in I}\left(\sum_{j=1}^m v_j(x) \ip{u_i, v_j}_{\H_k}\right)^2
        = \lVert A_{P,Q}\bm{v}_x\rVert_{\ell^2(I)}^2
        = \bm{v}_x^\top A_{P,Q}^* A_{P,Q}\bm{v}_x,
    \]
    where $\bm{v}_x = (v_1(x), \ldots, v_m(x))^\top \in\R^m$
    and $A_{P, Q}$
    is a linear operator $\R^m \to \ell^2(I)$
    given by $a=(a_1,\ldots,a_m)^\top \mapsto (\sum_{j=1}^m\ip{u_i, v_j}_{\H_k}a_j)_{i\in I}$,
    and $A_{P, Q}^*:\ell^2(I) \to \R^m$ is its dual (defined by the property $\ip{a, A_{P, Q}^* b}_{\R^m} = \ip{A_{P, Q}a, b}_{\ell^2(I)}$), which can be understood as the ``transpose'' of $A_{P, Q}$.
    % Note that the boundedness of $A_{P, Q}$ follows by
    % \[
    %     \lVert A_{P, Q}\bm{w} \rVert_{\ell^2(I)}
    %     = \sum_{i\in I}
    %     \left(\sum_{j=1}^m \ip{u_i, v_j}_{\H_k}w_j\right)^2
    %     \le \sum_{i\in I} \sum_{j=1}^m \ip{u_i, v_j}_{\H_k}^2 \lvert\bm{w}\rvert^2
    %     = \sum_{j=1}^m \sum_{i\in I} \ip{u_i, v_j}_{\H_k}^2 \lvert\bm{w}\rvert^2
    % \]
    Note that $A_{P, Q}^* A_{P, Q}$ can be regarded as an $m\times m$ matrix and we have
    \begin{equation*}
        (A_{P, Q}^* A_{P, Q})_{j, h}
        = \sum_{i\in I}\ip{u_i, v_j}_{\H_k}\ip{u_i, v_h}_{\H_k}
        = \ip{P v_j, P v_h}_{\H_k}.
        %\label{eq:mat-prod-apq}
    \end{equation*}
    We can also define $B_{P, Q}= A_{P^\perp, Q}$ by replacing $P$ with $P^\perp$.
    Then we have
    \[
        (A_{P, Q}^* A_{P, Q})_{j, h} + (B_{P, Q}^* B_{P, Q})_{j,h}
        = \ip{P v_j, P v_h}_{\H_k} + \ip{P^\perp v_j, P^\perp v_h}_{\H_k}
        = \ip{v_j, v_h}_{\H_k}
        = \delta_{jh},
    \]
    so $A^\top_{P, Q}A_{P, Q}$ is an $m\times m$
    positive semi-definite matrix
    with $A^\top_{P, Q}A_{P, Q} \le I_m$.
    
    It thus suffices to consider a uniform estimate of
    $\mu(\sqrt{\bm{v}_x^\top S \bm{v}_x}) - \mu_Z(\sqrt{\bm{v}_x^\top S \bm{v}_x})$
    with a positive semi-definite matrix $S \le I_m$.
    This $S$ can be written as $S=U^\top U$ by using a $U\in\R^{m\times m}$
    with $\lVert U\rVert_2\le 1$,
    so we shall solve the following problem:
    \begin{quote}
        Find a uniform upper bound of
        $\mu(\lVert U\bm{v}_x\rVert_2) - \mu_Z(\lVert U\bm{v}_x\rVert_2)$
        for any matrix $U\in\R^{m\times m}$ with $\lVert U\rVert_2 \le 1$.
    \end{quote}
    Now we can reduce our problem to a routine work of bounding the covering number of
    the function class $\F:=
    \{f_U:= x \mapsto \lVert U\bm{v}_x\rVert_2\mid U\in\mathcal{U}\}$,
    where $\mathcal{U}:=\{U\in\R^{m\times m}\mid \lVert U\rVert_2\le1\}$.
    
    For any $x\in\X$, we have
    \[
        \lVert \bm{v}_x \rVert_2^2 = \sum_{j=1}^\ell v_j(x)^2
        = \lVert Qk(\cdot, x)\rVert_{\H_k}^2
        \le \lVert k(\cdot, x)\rVert_{\H_k}^2
        = k(x, x).
    \]
    If $\mathcal{U}_\delta$ is a $\delta$-covering of $\mathcal{U}$,
    then $\{f_U\}_{U\in\mathcal{U}_\delta}$ gives a $\delta\sqrt{k_{\max}}$-covering.
    Indeed, for any $U, V\in \mathcal{U}$
    with $\lVert U -V \rVert_2\le\delta$, we have
    \[
        d_Z(f_U, f_V)^2
        =\frac1\ell\sum_{i=1}^\ell (\lVert U\bm{v}_{z_i}\rVert_2
        - \lVert V\bm{v}_{z_i}\rVert_2)^2
        \le \frac1\ell\sum_{i=1}^\ell \lVert (U - V)\bm{v}_{z_i}\rVert_2^2
        \le \delta^2 \frac1\ell\sum_{i=1}^\ell \lVert\bm{v}_{z_i}\rVert_2^2
        \le \delta^2k_{\max}.
    \]
    Here, we have the covering number bound
    $\log \N(\mathcal{U}, \lVert\cdot\rVert_2; \delta)\le m^2\log\left(1+\frac2\delta\right)$
    for $\delta\le1$ (and $0$ for $\delta\ge1$)
    as $\mathcal{U}$ can be seen as a unit ball of $\R^{m^2}$
    in a certain norm \citep[][Example 5.8]{hds-book},
    so $\log \N(\F, d_Z; \ve) \le m^2\log(1 + 2\sqrt{k_{\max}}/\ve)$
    for $\ve\le \sqrt{k_{\max}}$.
    
    Therefore, from Proposition \ref{prop:dudley}, we have
    \begin{align*}
        \mathcal{R}_Z(\F)
        &\le \frac{12}{\sqrt{\ell}}\int_0^{\sqrt{k_{\max}}} \sqrt{m^2 \log
        \left(1 + \frac{2\sqrt{k_{\max}}}{\ve}\right)}\dd\ve\\
        &=\frac{12m\sqrt{k_{\max}}}{\sqrt{\ell}}\int_0^1 \sqrt{\log
        \left(1 + \frac{2}{t}\right)}\dd t
        \le \frac{18m\sqrt{k_{\max}}}{\sqrt{\ell}},
    \end{align*}
    where we have used the estimate
    \[
        \int_0^1 \sqrt{\log
        \left(1 + \frac{2}{t}\right)}\dd t
        \le \int_0^1 \frac12\left(
            1 + \log\left(1 + \frac{2}{t}\right)
        \right)\dd t
        = \frac12 + \frac12\log\frac{27}4
        \le \frac32.
    \]
    
    Since we also have a bound $\lVert f_U\rVert_\infty \le \lVert U\rVert_2\sqrt{k_{\max}}$,
    we can use Proposition \ref{prop:ulln} to obtain
    \begin{align*}
        \mu(\lVert PQk(\cdot, x)\rVert_{\Hil_k}) - \mu_Z(\lVert PQk(\cdot, x)\rVert_{\Hil_k})
        \le \sup_{f\in\F}(\mu_Z(f) - \mu(f))
        \le \sqrt{\frac{k_{\max}}\ell}\left(36m + \sqrt{2\log\frac1\delta}\right)
    \end{align*}
    with probability at least $1-\delta$.
    So we have proven \eqref{eq:main-root-hp}.
    
    We next prove \eqref{eq:main-root-ex}
    by using Proposition \ref{prop:sh}.
    We have the same bound
    for $\log \N(\F, \lVert\cdot\rVert_{L^1}(\mu); \ve)$
    from the same argument as above,
    and so we especially get
    \[
        \log N\left(\F, \lVert\cdot\rVert_{L^1(\mu)}; \frac{\sqrt{k_{\max}}}{\ell}\right)
        \le m^2\log (1+2\ell).
    \]
    As $\lVert f\rVert_{L^\infty(\mu)}\le \sqrt{k_{\max}}=:F$ holds for all $f\in\F$,
    we can now apply Proposition \ref{prop:sh}
    with $\ve = F/\ell$ to obtain the desired conclusion.
\end{proof}

We next prove the following proposition that includes the desired assertion
by using Proposition~\ref{prop:aux}.
\begin{prop}\label{stronger}
    Let $Z=(z_i)_{i=1}^\ell$ be an $\ell$-point independent
    sample from $\mu$.
    Let $P$ be a random orthogonal projection in $\H_k$
    possibly depending on $Z$.
    For any integer $m\ge1$, with probability at least $1-\delta$,
    we have
    \[
        \int_\X \lVert Pk(\cdot, x)\rVert_{\Hil_k} \dd\mu(x)
        \le \frac1\ell\sum_{i=1}^\ell \lVert Pk(\cdot, z_i)\rVert_{\Hil_k}
        + \sqrt{\frac{k_{\max}}{\ell}}\left(36m + \sqrt{\frac92\log\frac2\delta}\right)
        + 3\sqrt{\sum_{j>m}\sigma_j}.
    \]
    Furthermore, in expectation, we have the following bound:
    \begin{align}
        \E{\int_\X \lVert Pk(\cdot, x)\rVert_{\Hil_k} \dd\mu(x)}
        &\le \E{\frac2\ell\sum_{i=1}^\ell \lVert Pk(\cdot, z_i)\rVert_{\Hil_k}}\nonumber\\
        &\quad + \frac{\sqrt{k_{\max}}}{\ell}\left(\frac{80m^2\log(1+2\ell)}9 + 69\right)
        + 4\sqrt{\sum_{j>m}\sigma_j}.
    \end{align}
\end{prop}
\begin{proof}
Note that we use the fact that for any projection operator $P$
$\lVert Pf\rVert \le \lVert f\rVert$ frequently within the proof.
For an $\ell$-point sample $Z = (z_1, \ldots, z_\ell) \subset \X$,
let us denote $\mu_Z$ be the mapping $f\mapsto \frac1\ell\sum_{i=1}^\ell f(z_i)$.
If we have $f_-,  f \in L^1(\mu)$ with $f_- \le f$,
we can generally obtain
\begin{align}
    \mu(f) - \mu_Z(f)
    &= (\mu(f) - \mu(f_-)) + (\mu(f_-) - \mu_Z(f_-))
     + (\mu_Z(f_-) - \mu_Z(f)) \nonumber\\
    &\le \mu(f - f_-) + (\mu(f_-) - \mu_Z(f_-)).
    \label{eq:main-root-1}
\end{align}
We here use $f(x) = \lVert Pk(\cdot, x)\rVert_{\Hil_k}$
and $f_-(x) = \lVert PP_mk(\cdot, x)\rVert_{\Hil_k} - 
\lVert PP_m^\perp k(\cdot, x)\rVert_{\Hil_k}$ for an $m$,
where $P_m$ is the projection operator onto $\mathop\mathrm{span}\{e_1,\ldots,e_m\}$
in $\H_k$
and $P_m^\perp$ is its orthogonal complement.
In this case,
$\mu(f - f_-)$ can easily be estimated by Cauchy--Schwarz as follows:
\begin{align}
    \mu(f - f_-)
    \le \mu(2\lVert PP_m^\perp k(\cdot, x)\rVert_{\Hil_k})
    &\le 2\mu(\lVert P_m^\perp k(\cdot, x)\rVert_{\Hil_k})\nonumber\\
    &\le 2\sqrt{\mu(\lVert P_m^\perp k(\cdot, x)\rVert_{\Hil_k}^2)}
    =2\sqrt{\sum_{j>m}\sigma_j},
    \label{eq:main-root-2}
\end{align}
where we have used the fact
\[
    \lVert P_m^\perp k(\cdot, x)\rVert_{\H_k}^2
    = \lVert k(\cdot, x)\rVert_{\H_k}^2
    - \lVert P_m k(\cdot, x)\rVert_{\H_k}^2
    = k(x, x) - \sum_{i=1}^m \sigma_ie_i(x)^2
    = \sum_{i=m+1}^\infty \sigma_ie_i(x)^2.
\]
%whose last equality holds almost surely with respect to $\mu$.
We also bound $\mu(f_-) - \mu_Z(f_-)$ by
\begin{equation}
    \mu(f_-) - \mu_Z(f_-)
    \le \mu(\lVert PP_mk(\cdot, x)\rVert_{\Hil_k}) - \mu_Z(\lVert PP_mk(\cdot, x)\rVert_{\Hil_k})
    +\mu_Z(\lVert P_m^\perp k(\cdot, x)\rVert_{\Hil_k}),
    \label{eq:main-root-3}
\end{equation}
where
we have used the second inequality in \eqref{eq:main-root-2} for $\mu_Z$.
The last term $\mu_Z(\lVert P_m^\perp k(\cdot, x)\rVert_{\Hil_k})$ above
is estimated
either in expectation or in high probability as follows:
\begin{equation}
    \begin{cases}
        \text{$\displaystyle\E{\mu_Z(\lVert P_m^\perp k(\cdot, x)\rVert_{\Hil_k})}
        \le \sqrt{\sum_{j>m}\sigma_j}$.}\\
        \text{$\displaystyle\mu_Z(\lVert P_m^\perp k(\cdot, x)\rVert_{\Hil_k})
        \le \sqrt{\sum_{j>m}\sigma_j} + 
        \sqrt{\frac{k_{\max}}{2\ell}\log\frac1\delta}$
        with probability at least $1-\delta$.}
    \end{cases}
    \label{eq:main-root-4}
\end{equation}
The latter follows from a simple calculation of Hoeffing's inequality.

Thus, it suffices to derive a bound for
$\mu(\lVert PP_mk(\cdot, x)\rVert_{\Hil_k}) - \mu_Z(\lVert PP_mk(\cdot, x)\rVert_{\Hil_k})$
or its expectation;
we do it by letting $Q=P_m$ and $\hat{f}=f$ in Proposition \ref{prop:aux}.
By combining (just summing up) the inequalities \eqref{eq:main-root-1}--\eqref{eq:main-root-4},
and \eqref{eq:main-root-hp}, 
we obtain the desired inequality in high probability. 
For the result in expectation,
we first combine the inequalities
\eqref{eq:main-root-1}--\eqref{eq:main-root-4},
and \eqref{eq:main-root-ex} to get the bound
\[
    \E{\mu(f)} - \E{\mu_Z(f)}
    \le
    \E{\mu_Z(\lVert PP_m k(\cdot, x)\rVert_{\H_k})}
    +\frac{\sqrt{k_{\max}}}{\ell}\left(\frac{80m^2\log(1+2\ell)}9 + 69\right)
    + 3\sqrt{\sum_{j>m}\sigma_j}
\]
(recall $f(x) = \lVert P k(\cdot, x)\rVert_{\H_k}$).
Since we can also estimate $\E{\mu_Z(\lVert PP_m k(\cdot, x)\rVert_{\H_k})}$
as
\begin{align*}
    \E{\mu_Z(\lVert PP_m k(\cdot, x)\rVert_{\H_k})}
    &\le \E{\mu_Z(\lVert P k(\cdot, x)\rVert_{\H_k})}
    + \E{\mu_Z(\lVert PP_m^\perp k(\cdot, x)\rVert_{\H_k})}\\
    &\le \E{\mu_Z(\lVert P k(\cdot, x)\rVert_{\H_k})}
    + \sqrt{\sum_{j>m}\sigma_j},
\end{align*}
we obtain the desired conclusion.
\end{proof}

\subsection{Proof of Remark \ref{rem:result}}\label{sec:proof-rem}
\begin{proof}
    We assume $\ell\ge 3$ here.
    Let $F(x):= - \beta^{-1}x^{1-1/d}\exp(-\beta x^{1/d})$.
    If $d\ge 2$, its derivative is
    \[
        F^\prime(x) = \exp(-\beta x^{1/d})  - \frac{1-1/d}\beta x^{-1/d}\exp(-\beta x^{1/d})
        = \left(1 - \frac{1-1/d}\beta x^{-1/d}\right)\exp(-\beta x^{1/d}).
    \]
    Thus, if $x\ge (\log\ell)^d/\beta^d$, we have $F^\prime(x) \ge d\exp(-\beta x^{1/d})$. This inequality is still true if $d=1$.
    By taking $m=\lfloor (2\log \ell)^d /\beta^d \rfloor$, we obtain
    \[
        \sum_{i>m}\sigma_i
        \lesssim \int_{2(\log\ell)^d / \beta^d}^\infty \exp(- \beta x^{1/d})\dd x
        \le - d F(2(\log \ell)^d / \beta^d)
        = \frac{2^{d-1}d}{\beta^d}\cdot \frac{(\log \ell)^{d-1}}{\ell^2}.
    \]
    Therefore, this choice of $m$ satisfies
    \[
        \sqrt{\sum_{i>m}\sigma_i} = \ord{\frac{(\log\ell)^{(d-1)/2}}\ell},
        \qquad
        m^2 = \ord{(\log\ell)^{2d}}.
    \]
    Combining these with the inequality in Corollary~\ref{cor:wce-iid} gives the desired estimate.
\end{proof}

\subsection{Proof of Proposition \ref{prop:wce-decomp}}
\begin{proof}
    We basically just compute the trace of the operator $P_Z^\perp\K$.
    Indeed, we have
    \begin{equation}
        \int_\X \lVert P_Z^\perp k(\cdot, x) \rVert_{\H_k}^2
        = \int_\X (k(x, x) - k^Z(x, x))\dd\mu(x),
        \label{eq:prop-tr-1}
    \end{equation}
    and, from \eqref{eq:mu-mercer}, we also have the following identity:
    \begin{equation}
        \int_\X k(x, x)\dd\mu(x) = \sum_{i=1}^\infty \ip{e_i, \K e_i}_{L^2(\mu)}.
        \label{eq:prop-tr-2}
    \end{equation}
    For $k^Z$,
    as we can write $k^Z(x, y) = \sum_{i=1}^\ell g_i(x)g_i(y)$
    by using $g_i\in L^2(\mu)$
    (see e.g., \eqref{eq:explicit-k-z}),
    we can also have
    \begin{equation}
        \int_\X k^Z(x, x)\dd\mu(x)
        = \sum_{i\in I} \ip{e_i, \K^Z e_i}_{L^2(\mu)}
        = \sum_{i=1}^\infty \ip{e_i, \K^Z e_i}_{L^2(\mu)},
        \label{eq:prop-tr-3}
    \end{equation}
    where $\K^Z:L^2(\mu)\to L^2(\mu)$ is the integral operator given by
    $g\mapsto \int_\X k^Z(\cdot, x)g(x)\dd\mu(x)$,
    and $(e_i)_{i\in I}$ is an orthonormal basis of $L^2(\mu)$
    including $(e_i)_{i=1}^\infty$.
    The second equality follows from the fact that
    $\K - \K^Z$ is a (semi-)positive definite operator since $k-k^Z$ is a positive definite kernel,
    and so we have $0\le \ip{e_i, \K^Z e_i}_{L^2(\mu)} \le \ip{e_i, \K e_i}_{L^2(\mu)} = 0$
    for any $i\in I \setminus\mathbb{Z}_{>0}$.
    For this integral operator,
    since we have $k^Z(\cdot, x) = P_Z k(\cdot, x)$,
    we can prove
    \[
        \K^Z g = \int_\X P_Z k(\cdot, x)g(x)\dd\mu(x)
        = P_Z \int_\X k(\cdot, x)g(x)\dd\mu(x)
        = P_Z \K g
    \]
    for any $g\in L^2(\mu)$ under the well-definedness of $\K$.
    Thus, from \eqref{eq:prop-tr-1}--\eqref{eq:prop-tr-3}, we have
    \begin{equation}
        \int_\X \lVert P_Z^\perp k(\cdot, x) \rVert_{\H_k}^2
        = \sum_{i=1}^\infty \ip{e_i, (\K - \K^Z)e_i}_{L^2(\mu)}
        = \sum_{i=1}^\infty \ip{e_i, P_Z^\perp \K e_i}_{L^2(\mu)}.
        \label{eq:prop-tr-4}
    \end{equation}

    For general $f\in\H_k$ and $g\in L^2(\mu)$, we can prove
    \[
        \ip{f, \K g}_{\H_k}
        = \ip{f, \int_\X k(\cdot, x)g(x)\dd\mu(x)}_{\H_k}
        = \int_\X \ip{f, k(\cdot, x)}_{\H_k} g(x)\dd\mu(x)
        = \ip{f, g}_{L^2(\mu)},
    \]
    so that in particular
    \[
        \ip{g, P_Z^\perp \K g}_{L^2(\mu)}
        = \ip{\K g, P_Z^\perp \K g}_{\H_k}
        = \lVert P_Z^\perp \K g \rVert_{\H_k}^2.
    \]
    By letting $g=e_i$ in the above equation,
    we can deduce the desired equality from \eqref{eq:prop-tr-4}.
    For the inequality, use the bound
    \[
        \lVert P_Z^\perp \K e_i\rVert_{\H_k}^2
        \le \lVert \K e_i\rVert_{\H_k}^2 = \lVert \sigma_i e_i\rVert_{\H_k}^2
        = \sigma_i \lVert \sqrt{\sigma_i} e_i\rVert_{\H_k}^2 = \sigma_i
    \]
    for each $i>m$.
\end{proof}

\subsection{Proof of Corollary \ref{cor:dpp}}
\begin{proof}
    From Proposition \ref{prop:wce-decomp}
    and \eqref{eq:weighted-wce},
    it suffices to prove for an arbitrary $g\in L^2(\mu)$ that
    \[
        \lVert P_Z^\perp \K g\rVert_{\H_k}^2
        = \inf_{w_i}\sup_{\lVert f\rVert_{\H_k}\le1}\left\lvert \mu(fg)-
        \sum_{i=1}^\ell w_if(z_i)\right\rvert^2
        \le 4\sum_{i>\ell}\sigma_i.
    \]
    It is indeed an immediate consequence of \citet[][Theorem 4]{bel21}.
\end{proof}

\subsection{Proof of Lemma \ref{rem:square-kernel}}
\begin{proof}
    Given the Mercer decomposition
    $k(x, y) = \sum_{i=1}^\infty \sigma_i e_i(x)e_i(y)$, we can compute
    \begin{align*}
        h_\mu(x, y)
        &=\int_\X k(x, t)k(t, y)\dd\mu(t)\\
        &=\sum_{i,j=1}^\infty \sigma_i\sigma_j e_i(x)e_j(y)\int_\X e_i(t)e_i(t)\dd\mu(t)\\
        &=\sum_{i,j=1}^\infty \delta_{ij} \sigma_i\sigma_j e_i(x)e_j(y)
        =\sum_{i=1}^\infty \sigma_i^2 e_i(x)e_i(y),
    \end{align*}
    where we have used the fact
    that $(e_i)_{i=1}^\infty$ is an orthonormal set in $L^2(\mu)$.
\end{proof}

\subsection{Proof of Lemma \ref{lem:or-eig}}
\begin{proof}
    From \eqref{eq:ip-z}, we have
    \begin{equation}
        \ip{f_i, f_j}_{L^2(\mu)}
        = v_i^\top (H^+)^\top H^\top H H^+ v_j
        = (HH^+ v_i)^\top (HH^+ v_j).
        \label{eq:k-z-ortho}
    \end{equation}
    Here, note that $\{v_i,\, \kappa_i>0\}\subset (\ker H^\top)^\perp$ as we have,
    for any $v\in\ker H^\top$,
    \[
        0 = v^\top Hk(Z, Z)^+H^\top v
        = \sum_{i=1}^\ell \kappa_i v^\top v_i v_i^\top v
        = \sum_{i=1}^\ell \kappa_i (v^\top v_i)^2.
    \]
    Therefore, $HH^+ v_i = v_i$ if $\kappa_i>0$
    since $HH^+$ is the projection onto $(\ker H^\top)^\perp$,
    and so $\{f_i,\, \kappa_i>0\}$ is orthonormal from \eqref{eq:k-z-ortho}.
    We can also see that $f_i = (H^+v_i)^\top k(Z, \cdot)$
    is an eigenfunction of $\K^Z$
    from the remark below \eqref{eq:k-z} and $HH^+v_i=v_i$.
\end{proof}

\subsection{Proof of Proposition \ref{prop:ae}}
\begin{proof}
    We rewrite $k_\mu^Z$ in terms of another summation as follows:
    \begin{align}
        k^Z_\mu(x, y)
        &:= \sum_{i=1}^\ell \kappa_i f_i(x)f_i(y) \nonumber\\
        &= k(x, Z)H^+ \left(\sum_{i=1}^\ell \kappa_i v_iv_i^\top\right) (H^\top)^+k(Z, y)
        \nonumber\\
        &= k(x, Z)H^+ H k(Z, Z)^+ H^\top (H^\top)^+k(Z, y)\nonumber\\
        &= \sum_{\lambda_i >0} \frac1{\lambda_i}
        u_i^\top H^\top (H^+)^\top k(Z, x) k(y, Z) H^+ H u_i,
        \label{eq:def-k-z-mu}
    \end{align}
    where $(\lambda_i, u_i)$ are eigenpairs of $k(Z, Z)$.
    Recall also that we have
    \begin{equation}
        k^Z(x, y) = k(x, Z)k(Z, Z)^+k(Z, y)
        =\sum_{\lambda_i>0}\frac1{\lambda_i}u_i^\top k(Z, x)k(y, Z)u_i.
        \label{eq:express-k-z}
    \end{equation}
    From \eqref{eq:def-k-z-mu} and this,
    it suffices to prove $u^\top k(Z, \cdot)
    = u^\top H^\top (H^+)^\top k(Z, \cdot)$
    in $L^2(\mu)$
    for any $u \in\R^\ell$.
    Indeed,
    we have
    \begin{align*}
        &\int_\X \left(u^\top k(Z, x) - u^\top H^\top (H^+)^\top k(Z, x) \right)^2\dd\mu(x)\\
        &=\int_\X \left(u^\top \left(I_\ell - H^\top (H^+)^\top \right) k(Z, x)\right)^2\dd\mu(x)\\
        &=u^\top \left(I_\ell - H^\top (H^+)^\top \right)
        \left(\int_\X k(Z, x)k(x, Z)\dd\mu(x)\right)
        (I_\ell - H^+H)u\\
        &=u^\top \left(I_\ell - H^\top (H^+)^\top \right)
        H^\top H
        (I_\ell - H^+H)u=0
    \end{align*}
    since
    $H^\top (H^+)^\top H^\top = H^\top$
    and $HH^+H=H$ hold ($I_\ell$ is the identity matrix).
    Thus, we obtain the desired assertion.

    Finally, we prove that $k_\mu^Z$ and $k^Z$ coincide when
    $\ker h_\mu(Z, Z) \subset \ker k(Z, Z)$.
    From \eqref{eq:def-k-z-mu} and \eqref{eq:express-k-z},
    it suffices to prove $H^+H u_i = u_i$ for indices $i$ with $\lambda_i>0$.
    Note that $H^+H$ is the orthogonal projection onto the orthogonal complement
    of $\ker H = \ker H^\top H = h_\mu(Z, Z)$ from a general property of the pseudo-inverse.
    Since $u_i$ is an eigenvector of $k(Z, Z)$ with a positive eigenvalue $\lambda_i$,
    it is orthogonal to any $v\in\ker k(Z, Z)$
    (as $u_i^\top v = \lambda_i^{-1} u_i^\top k(Z, Z) v = 0$).
    Therefore, if we have $\ker h_\mu(Z, Z) \subset \ker k(Z, Z)$,
    $u_i$ is also orthogonal to $\ker h_\mu(Z, Z)$
    and so $H^+H u_i = u_i$ as desired.
\end{proof}

\subsection{Proof of Proposition \ref{prop:emp-eig-decay-x}}
First, we give a proof for a folklore property of
products of positive semi-definite matrices.
\begin{lem}
    Let $\ell, m\ge n$ be positive integers
    and $A, B\in\R^{n\times n}$ be (symmetric) positive semi-definite matrices.
    Assume $B = C^\top C = D^\top D$ for a real matrix $C\in\R^{m\times n}$ and
    $D\in\R^{\ell\times n}$.
    Then, $CAC^\top$ and $DAD^\top$ have the same set of nonzero eigenvalues with the same multiplicity (in terms of real eigenvectors).
\end{lem}
\begin{proof}
    For a real square matrix $M\in\mathbb{R}^{j\times j}$ and a real number $\lambda$,
    let us define $S_\lambda(M):=\{v\in\R^j\mid Mv = \lambda v\}$
    be the real eigenspace of $M$ corresponding to $\lambda$.

    We shall prove there is a bijection between $S_\lambda(AB)$ and $S_\lambda(CAC^\top)$
    for each real $\lambda\ne0$ (and the same for $S_\lambda(DAD^\top)$ by symmetry).
    Once we establish this,
    we see that each $\lambda\ne0$ has the same multiplicity as an eigenvalue of
    $CAC^\top$ and $DAD^\top$
    (multiplicity can be zero; in that case $\lambda$ is not an eigenvalue),
    and the desired assertion follows.

    Let us fix $\lambda\ne0$.
    If $v\in S_\lambda(CAC^\top)$, we have
    $CAC^\top (Cv) = CABv = \lambda (Cv)$,
    so $Cv\in S_\lambda(CAC^\top)$.
    We also have $Cv^\prime \ne Cv$ for another element $(v\ne) v^\prime\in S_\lambda(AB)$
    since $AC^\top(Cv^\prime - Cv) = AB(v^\prime - v) = \lambda (v^\prime -v) \ne 0$.
    Thus, matrix multiplication by
    $C$ is an injective map from $S_\lambda(AB)$ to $S_\lambda(CAC^\top)$.

    Let us finally prove $S_\lambda(AB)\ni v \mapsto Cv \in S_\lambda(CAC^\top)$ is surjective.
    Let $u\in S_\lambda(CAC^\top)$.
    Then, $u = \lambda^{-1} (\lambda u) = \lambda^{-1} CAC^\top u = C (\lambda^{-1}AC^\top u)$,
    so we can write $u = Cv$ for $v = \lambda^{-1}AC^\top u$.
    It remains to prove $v\in S_\lambda(AB)$,
    but we can see it as follows:
    \[
        ABv = AB\left(\frac1\lambda AC^\top u\right)
        =\frac1\lambda (AC^\top C)AC^\top u
        =\frac1\lambda AC^\top (CAC^\top u)
        =\frac1\lambda AC^\top (\lambda u)
        =\lambda v.
    \]
    Therefore, we have a bijection between $S_\lambda(AB)$ and $S_\lambda(CAC^\top)$
    and we are done.
\end{proof}

Recall $\mu(k_\mu^Z - k_{s,\mu}^Z)
\le \sum_{i=s+1}^\ell\kappa_i$ holds
for eigenvalues $\kappa_1\ge\cdots\kappa_\ell\ge0$ of  $H_\mu k(Z, Z)^+H_\mu^\top$
with $H_\mu^\top H_\mu = h_\mu(Z, Z)$
(that immediately follows from the definitions
of $k_\mu^Z$ and $k_{s,\mu}^Z$, and that $f_i$ are $L^2(\mu)$-orthonormal).
By replacing $\mu$ with $\mu_X$,
we have $\mu_X(k_X^Z - k_{s, X}^Z) \le \sum_{i=s+1}^\ell \kappa_i^X$
for eigenvalues of $\kappa_1^X \ge \cdots \ge \kappa_\ell^X\ge 0$ of $H_X k(Z, Z)^+H_X^\top$,
where $H_X^\top H_X = h_X(Z, Z) = \frac1M k(Z, X)k(X, Z)$.

By using the lemma, we can see that
$\kappa_i^X$ are actually the same as the eigenvalues of
$\frac1M k(X, Z)k(Z, Z)^+k(Z, X) = \frac1M k^Z(X, X)$.
As $k - k^Z$ is a positive definite kernel,
$k(X, X) - k^Z(X, X)$ is a positive semi-definite matrix,
the $i$-th largest eigenvalue of $k^Z(X, X)$
is bounded by the $i$-th largest eigenvalue of $k(X, X)$ (Weyl's inequality).

Now, let $\lambda_1^X \ge \lambda_2^X \ge \cdots \ge 0$
be the eigenvalues of $k(X, X)$.
From the above argument,
we have
\[
    \mu_X(k_X^Z - k_{s, X}^Z)
    \le \sum_{i=s+1}^\ell \kappa_i^X
    \le \frac1M\sum_{i=s+1}^\ell \lambda_i^X
    \le \frac1M\sum_{i=s+1}^M \lambda_i^X.
\]
Notice that we can apply Lemma \ref{lem:trace} with $X$ instead of $Z$,
and obtain $\E{\mu_X(k_X^Z - k^Z_{s, X})} \le \sum_{i>s} \sigma_i$
as desired.

\subsection{Proof of Proposition \ref{thm:k-s-x}}

\begin{proof}
    Fix a sample $X$ with $\ker k(X, Z)\subset \ker k(Z, Z)$
    and let us use the same notation as in $\mu$,
    i.e.,
    \begin{itemize}
        \item $H^\top H = h_X(Z, Z) = \frac1Mk(Z, X)k(X, Z)$;
        \item $Hk(Z, Z)^+H^\top = V\mathop\mathrm{diag}(\kappa_1, \ldots, \kappa_\ell) V^\top$
                with $\kappa_1\ge\cdots\kappa_\ell\ge0$
                and $V$ being orthogonal;
        \item $f_i = (H^+v_i)^\top k(Z, \cdot)$
        and $k_X^Z(x, y) = \sum_{i=1}^\ell \kappa_i f_i(x)f_i(y)$.
    \end{itemize}

    In this case, from the same argument as the last paragraph
    in the proof of Proposition \ref{prop:ae},
    we have $H^+H$ is an identity map over $(\ker h_X(Z, Z))^\perp
    = (\ker k(X, Z))^\perp \supset (\ker k(Z, Z))^\perp$.
    By considering the SVD of $k(Z, Z)$, we see that $(\ker k(Z, Z))^\perp$
    is exactly the linear subspace of $\R^\ell$ spanned by eigenvectors of $k(Z, Z)$
    with nonzero eigenvalues, which is equal to $\{k(Z, Z)v \mid v\in\R^\ell\}
    = \{k(Z, Z)^+v \mid v\in\R^\ell\}$.
    In particular, we have $H^+H k(Z, Z)^+ = k(Z, Z)^+$.

    We now prove that $\{\sqrt{\kappa_i} f_i \mid i\ge 1,\, \kappa_i>0\}$ actually forms
    an orthonoramal set in $\H_k$.
    Indeed, if $\kappa_i, \kappa_j>0$,
    we have
    \begin{align*}
        \ip{\sqrt{\kappa_i}f_i, \sqrt{\kappa_j}f_j}_{\H_k}
        &= \sqrt{\kappa_i\kappa_j} v_i^\top (H^+)^\top k(Z, Z) H^+ v_j\\
        &= \frac1{\sqrt{\kappa_i\kappa_j}}
            v_i^\top \left[ H k(Z, Z)^+ H^\top \right] (H^+)^\top k(Z, Z) H^+
            \left[ H k(Z, Z)^+ H^\top \right] v_j\\
        &= \frac1{\sqrt{\kappa_i\kappa_j}} v_i^\top H k(Z, Z)^+ k(Z, Z) k(Z, Z)^+ H^\top v_j\\
        &= \frac1{\sqrt{\kappa_i\kappa_j}} v_i^\top H k(Z, Z)^+ H^\top v_j
        = \delta_{ij},
    \end{align*}
    where we have used the fact that $v_i$ and $v_j$ are eigenvectors of $H k(Z, Z)^+ H^\top$
    with eigenvalues $\kappa_i$ and $\kappa_j$, respectively.

    Let $P:\H_k\to\H_k$ be the orthogonal projection onto
    $\mathop\mathrm{span}\{\sqrt{\kappa_i}f_i \mid i>s,\, \kappa_i>0\}$.
    Then, we have
    \begin{align*}
        Pk(\cdot, x) = \sum_{i=s+1}^\ell
        \ip{\sqrt{\kappa_i}f_i, k(\cdot, x)}_{\H_k}\sqrt{\kappa_i}f_i
        = \sum_{i=s+1}^\ell
        \sqrt{\kappa_i}f_i(x)\sqrt{\kappa_i}f_i,
    \end{align*}
    and so $\lVert Pk(\cdot, x)\rVert_{\H_k}^2
    = \sum_{i=s+1}^\ell \kappa_if_i(x)^2 = k^Z(x, x) - k_{s, X}^Z(x, x)$.
    Note that the projection $P$ is a random operator depending on the sample $X$.
    Now, we can use Theorem \ref{thm:low-dim-proj} with
    the empirical measure given by $X$ instead of $Z$ to obtain
    \begin{equation}
        \E{\mu(\sqrt{k^Z - k^Z_{s, X}})}
        \le 2\E{\mu_X(\sqrt{k_X^Z - k^Z_{s, X}})}
        + 4\sqrt{\sum_{i>m}\sigma_i} 
        + \frac{\sqrt{k_{\max}}}{M}\left(\frac{80m^2\log(1+2M)}9 + 69\right).
        \label{eq:final-k-s-x}
    \end{equation}    
    for any integer $m\ge1$,
    where we have used $\lVert Pk(\cdot, x)\rVert_{\H_k} = \sqrt{k^Z(x, x) - k^Z_{s, X}(x, x)}
    = \sqrt{k^Z_X(x, x) - k^Z_{s, X}(x, x)}$ almost surely.
    From Proposition \ref{prop:emp-eig-decay-x},
    we have
    \[
        \E{\mu_X(\sqrt{k_X^Z - k^Z_{s, X}})}^2
        \le \E{\mu_X(\sqrt{k_X^Z - k^Z_{s, X}})^2}
        \le \E{\mu_X(k_X^Z - k^Z_{s, X})}
        \le \sum_{i>s} \sigma_i,
    \]
    and combining it with \eqref{eq:final-k-s-x} leads to the desired conclusion.
\end{proof}

\subsection{Proof of Theorem \ref{thm:main-kq-emp}}

\begin{proof}
    We first prove the result for $Q_n = \mathrm{KQuad(k_{s, Y}, Y)}$.
    Since $k(x, x) \ge k^Z(x, x) = k^Z_Y(x, x) \ge k_{s, Z}^Y(x, x)$
    for $x\in Y$ from Proposition \ref{prop:discrete},
    we have
    \[
        \mu_Y(\sqrt{k - k_{s,Y}^Z})
        \le \mu_Y(\sqrt{k - k^Z}) + \mu_Y(\sqrt{k_Y^Z - k_{s, Y}^Z}).
    \]
    From Proposition \ref{prop:emp-eig-decay-x},
    by taking the expectation with regard to $Y$,
    we have
    \[
        \E{\mu_Y(\sqrt{k_Y^Z - k_{s, Y}^Z})}
        \le \sqrt{\E{\mu_Y(k_Y^Z - k_{s, Y}^Z)}}
        \le \sqrt{\sum_{i>s}\sigma_i},
    \]
    and so we obtain
    \[
        \E{\mu_Y(\sqrt{k - k_{s, Y}^Z})}
        \le \mu(\sqrt{k - k^Z}) + \sqrt{\sum_{i>s}\sigma_i}
    \]
    By combining it with \eqref{eq:estimate-empirical},
    it is now sufficient to show
    $\E{\mmd_k(\mu_Y, \mu)} \le \sqrt{c_{k,\mu}/N}$,
    but actually it follows from the identity
    $\E{\mmd_k(\mu_Y, \mu)^2} = c_{k, \mu}/N$,
    which can be shown by a straightforward calculation
    \citep[see, e.g.,][Proof of Theorem 7]{hayakawa21b}.

    In the case of $Q_n = \mathrm{KQuad}(k_{s,\mu}^Z, Y)$,
    we instead have the decomposition
    \[
        \mu_Y(\sqrt{k - k_{s,\mu}^Z})
        \le \mu_Y(\sqrt{k - k^Z}) + \mu_Y(\sqrt{k_\mu^Z - k_{s, \mu}^Z});
    \]
    Theorem~\ref{thm:k-s-mu-z-decay}
    yields the desired estimate for expectation.
\end{proof}

\begin{comment}

\end{comment}

\end{document}